%% file: artigodoublejumparxiv.tex
\documentclass[12pt,twoside]{article}
\usepackage{amsmath,amsthm,amssymb,amsfonts}
\usepackage{graphicx}
\usepackage{hyperref}
\usepackage{enumerate}

\theoremstyle{cplain}
\newtheorem{Thrm}{Theorem}[section]
	\newtheorem{Prop}[Thrm]{Proposition}
	\newtheorem{Cor}[Thrm]{Corollary}
	\newtheorem{Lemma}[Thrm]{Lemma}

{
\theoremstyle{cremark}

}
{
\theoremstyle{uremark}

}
{
\theoremstyle{cdefinition}
\newtheorem{Def}[Thrm]{Definition}
}

\newcommand{\ehf}{\text{EHF}}

\newcommand{\val}{\text{VAL}}
\newcommand{\pat}{\text{PAT}}
\newcommand{\bb}{\text{BB}}
\newcommand{\bc}{\text{BC}}

\graphicspath{{figuras/}}

\begin{document}
\title{Spaces of completions of elementary theories and convergence laws for random hypergraphs}
\author{Nicolau C. Saldanha, M\'arcio Telles}
\maketitle



\abstract{%

Consider the binomial model $G^{d+1}(n,p)$ of the random
$(d+1)$-uniform hypergraph on $n$ vertices, where each edge is present,
independently of one another,
with probability $p:\mathbb{N}\to[0,1]$. We prove that, for all logarithmo-exponential $p\ll n^{-d+\epsilon}$, the probabilities of all
elementary properties of hypergraphs converge, with particular emphasis in the ranges
$p(n)\sim C/n^d$ and $p(n) \sim C\log(n)/n^d$. The exposition is unified by constructing, for each such function $p$,
the topological space of all completions of its almost sure theory. This space turns out to be compact, metrizable and totally disconnected, but further properties depend
 on the range of $p$. The convergence of the probabilities of elementary properties 
is associated with a borelian probability measure on the space.

}

                                          \section{Introduction}

    It has now been more than fifty years since Erd\H os and R\'enyi laid the foundations for the
    study of random graphs on their seminal paper \emph{On the evolution of random graphs}
    \cite{erdos2}, where they considered the binomial random graph model $G(n,p)$. This consists of
     a graph on $n$ vertices
    where each of the potential ${n\choose 2}$ edges is present with probability $p$, all these events
    being independent of each other. Many interesting asymptotic questions arise when $n$ tends to $\infty$
    and we let $p$ depend on $n$. 
    
   Among other results, they showed that many properties of graphs exhibit a \emph{threshold
   behavior}, meaning that the probability that the property holds on $G(n,p)$ turns from near $0$ to near $1$ in a 
   narrow range of the edge probability $p$. More concretely, given a property $P$ of graphs, in
   many cases there is a \emph{threshold} function $p:\mathbb{N}\to[0,1]$ such that, as $n\to\infty$, the probability that $G(n,\tilde{p})$
   satisfies $P$ tends to $0$ for all $\tilde{p}\ll p$ and tends to $1$ for all $\tilde{p}\gg p$. Erd\H os
   and R\'enyi showed, for example, that the threshold for connectivity is $p=\frac{\log n}{n}$.
   They also showed that there is a profound change in the component structure of $G(n,p)$ for 
   $p$ around $\frac{1}{n}$, where one of its many connected components suddenly becomes
   dramatically larger than all others, a phenomenon mainly understood today as a phase transition. The range of $p$ where this occur is called the \emph{Double Jump} and 
   has received enormous attention from researchers since then.
   
   The above threshold behaviors suggest that one could expect to describe some convergence
   results, where the probabilities of all properties in a certain class converge to 
   known values as $n\to\infty$. Among the first convergence results there are the \emph{zero-one laws}, where all properties of graphs expressible by a first order formula (called \emph{elementary properties}) converge to $0$ or 
   to $1$. This happens, for example, if $p$ is independent of $n$. Many other instances of zero-one laws for random graphs were obtained by Joel Spencer in the book \emph{The Strange Logic of Random Graphs} \cite{spencer}. There he shows that zero-one laws hold if $p$ lies between a number of ``critical" functions. More concretely, if $p$ satisfies one of the following conditions

   \begin{enumerate}
   
      \item $p\ll n^{-2}$
      \item $n^{-\frac{1+l}{l}}\ll p\ll n^{-\frac{2+l}{1+l}}$ for some $l\in\mathbb{N}$
      \item $n^{-1-\epsilon}\ll p\ll n^{-1}$ for all $\epsilon>0$
      \item $n^{-1}\ll p\ll (\log n)n^{-1}$
      \item $(\log n)n^{-1}\ll p\ll n^{-1+\epsilon}$ for all $\epsilon>0$
   
    \end{enumerate}
    then a zero-one law holds.
        
    Note that clause (b) is, in fact, a \emph{scheme} of clauses. Note also that (a) can be viewed
    as a special case of (b). There are functions $p\ll n^{-1+\epsilon}$ not considered by 
    any of the above conditions. Such ``gaps" occur an infinite number of times near the critic
    functions in the scheme (b) and two more times: one between clauses (c) and (d) and the 
    other between clauses (d) and (e). Spencer shows that, for some functions $p$ conveniently near that ``critic" functions in (b) and the critic function $\frac{\log n}{n}$, corresponding to the gap (d)-(e), the probabilities 
    of all elementary properties converge to constants $c\in[0,1]$ as
    $n\to\infty$. This situation, more general than that of a zero-one law, is called a \emph{convergence law}. Spencer's book also has a brief discussion of some
    functions on the gap (c)-(d). This gap is traditionally known as \emph{the Double Jump}.
    
    One sees immediately that the possibility that an edge probability function could oscillate infinitely often between two 
    different values can be an obstruction to getting convergence laws. With this difficulty in mind,
    we consider the edge probability functions $p:\mathbb{N}\to[0,1]$ that belong to Hardy's class
    of \emph{logarithmo-exponential functions}. This class is entirely made of eventually monotone 
    functions, avoiding the above mentioned problem, and has the additional convenience of being closed by elementary algebraic operations and compositions that can involve
    logarithms and exponentials. All thresholds of natural properties of graphs seem to belong     
    in Hardy's class.
    
    Generally speaking, our arguments imply that, once one restricts the edge probabilities
    to functions in Hardy's class, there are no further ``gaps": all logarithmo-exponential edge 
    probabilities $p\ll n^{-1+\epsilon}$ are convergence laws. The arguments in Spencer's book
    are sufficient for getting most of these convergence laws, except for those in the window
    $p\sim C\frac{\log n}{n}$ , $C>0$, where just the value $C=1$ is discussed.

    There are three main interests on this work.
    The first one lies in the completion of the discussion of the convergence laws in the window $p\sim C\frac{\log n}{n}$ for other values of $C$. We will see that this window hides an infinite collection of zero-one and convergence laws and that those can be presented in a simple and organized fashion.    
     The second is to present a complete and detailed discussion of the convergence laws in the double jump, much in the 
    spirit of the sketch given in Spencer's book.     
    The third is the language used to present some arguments and results: We were led to associate to every edge probability $p(n)$ a compact topological space, the space of completions, with a unit borelian measure, generalizing the concept of complete set of completions, found in Spencer's book. The structure of this space 
    depends on the range in consideration: it can vary from a simple one-point space to a countably infinite set and, in the most interesting case of the Double Jump, to a 
    Cantor space. 
     Moreover, we present the arguments in the more general framework of random uniform hypergraphs, get simple axiomatizations of the almost sure theories involved and describe all elementary
    types of the countable models of these theories.
    
    Convergence laws have a deep connection with some elementary concepts of logic. More 
    precisely, zero-one laws occur when the class $\Theta_p$ of almost sure elementary properties is 
    \emph{complete}.
    Some everyday results in first-order logic imply that when all countably infinite models of $\Theta_p$   are \emph{elementarily equivalent} (that is, satisfy the same elementary properties), $\Theta_p$ is complete. This is obviously the case if there is, apart from
    isomorphism, only one such model: in this case, we say that $\Theta_p$ is $\aleph_0$-categorical.
    We will face situations where the countable models of
    $\Theta_p$ are, indeed, unique up to isomorphism. In other cases, the almost sure theory is still complete but the countable models are not unique: in the
    instances of the latter situation, the countable models are not far from being uniquely determined and, in particular, lend themselves to an exhaustive characterization. 
    
The more general context of convergence laws requires examination of the space of all possible completions of the almost sure theory $\Theta_p$.
  Here, $\Theta_p$ is not necessarily complete but is not very far from that and we still manage to 
    classify their countable models. At this stage, it is convenient to put a topology on the space $\mathcal{K}(\Theta_p)$ of all possible completions of $\Theta_p$. This process is presented in section \ref{completions}, which is mostly
   independent.

Let $J$ be the class of all $L$-functions $p$ such that, for all $\epsilon>0$, $p\ll n^{-d+\epsilon}$. 
We prove the following theorem.

 \begin{Thrm}   \label{grandeteorema}
 
 Every $p\in J$ is a convergence law.  More precisely:

\begin{enumerate}[(i)]

\item{

 If the $L$-function $p$ satisfies one of the following conditions
 
 \begin{enumerate}

   \item $0\le p\le n^{-(d+1)}$
   \item $n^{-\frac{1+ld}{l}}\ll p\ll n^{-\frac{1+(l+1)d}{l+1}}$, for some $l\in\mathbb{N}$
   \item $n^{-(d+\epsilon)}\ll p\ll n^{-d}$ for all $\epsilon>0$
   \item $n^{-d}\ll p\ll(\log n)n^{-d}$
   \item $p\sim C\cdot\frac{\log n}{n^d}$, where $\frac{d!}{v^*+1}<C<\frac{d!}{v^*}$ for  some $d,v^*\in\mathbb{N}$.  
                         \item $p\sim \frac{d!}{v^*}\cdot\frac{\log n}{n^d}$ where $\omega\to\pm\infty$ or
            $\omega\to C$ where $l-1<C<l$, for some $l\in\mathbb{N}$
    \item $p\sim\frac{d!}{v^*}\cdot\frac{\log n+l\log\log n}{n^d}$, where $c\to\pm\infty$           
    \item $(\log n)n^{-d}\ll p\ll n^{-d+\epsilon}$                               

\end{enumerate} 
 then $p$ is a zero-one law and, in particular, the corresponding $\mathcal{K}(\Theta_p)$ is a one-point space.
 
}

\item{
 
 If the $L$-function $p$ satisfies one of the following conditions
 
 \begin{enumerate}
 
   \item $p\sim c\cdot n^{-\frac{1+ld}{l}}$, for some constant $c\in(0,+\infty)$
   \item $p\sim\frac{d!}{v^*}\cdot\frac{\log n+l\log\log n+c(n)}{n^d}$, where $c\in\mathbb{R}$

 \end{enumerate}
then $p$ is a convergence law and the corresponding $\mathcal{K}(\Theta_p)$ is infinite countable.

}

\item{

If 
   $p\sim\frac{\lambda}{n^d}\text{, for some $\lambda\in\mathbb{R}$}$
then $p$ is a convergence law and the corresponding $\mathcal{K}(\Theta_p)$ is a Cantor space.

}

\end{enumerate}

 \end{Thrm}   
    
We conjecture that $J$ is a maximal interval for the property of being entirely
       made of convergence laws.
    Indeed, Spencer, in \cite{spencer}, shows that if $\alpha=1/3$ then $n^{-\alpha}$ fails to be a convergence law for the random graph $G(n,p)=G^{2}(n.p)$. His methods
       seem to apply to other rational values of $\alpha\in(0,1)$ and $d\in\mathbb{N}$.  
                 
 This is an extended version of the second author's thesis \cite{tese} presented to PUC-Rio, in September 2013,  incorporating some suggestions from the examining commission. 
Among such extensions of the original text,
there is a discussion of convergence laws for the random hypergraph in the range
    $p\sim\frac{\lambda}{n^d},$
called the \emph{Double Jump}.

The first author is partially supported by grants from CAPES, CNPq and FAPERJ (Brazil). The second author is partially supported by UERJ (Brazil). Both 
authors would like acknowledge the valuable
remarks by J. Freire (PUC-Rio), P. Gon\c calves (PUC-Rio), Y. Kohayakawa (USP), R. Morris (IMPA), R. I. Oliveira (IMPA), J. Spencer (NYU) and C. Tomei (PUC-Rio).

                                                \section{Preliminaires}

          \subsection{The Model $G^{d+1}(n,p)$}

Let $p:\mathbb{N}\to[0,1]$.
We consider the binomial model $G^{d+1}(n,p)=G^{d+1}(n,p(n))$ of random $(d+1)$-uniform  hypergraphs on $n$ vertices with probability $p(n)$. 
For fixed $n\in\mathbb{N}$, it is a finite probability space consisting of all hypergraphs on $n$ labelled vertices, where each edge is a set of cardinality $d+1$; 
for $H$ such a hypergraph with $k$ edges, we have                                                               
$$\mathbb{P}_n[H]=p^k(1-p)^{{n\choose d+1}-k}.$$
 Another characterization of this probability space is to insist that each of the potential ${n\choose d+1}$ 
 edges be present in $G^{d+1}(n,p)$ with probability $p$, each of these events being independent of each other.
We will, more often, prefer the latter because it is more convenient in applications.
Notice that, in the display above, we write $\mathbb{P}_n[H]$ (instead of $\mathbb{P}[H]$) to stress the dependence on $n$: this will often be done.

Our interest lies on the asymptotic behavior of $G^{d+1}(n,p)$ when $n\to\infty$. 
More formally, a \emph{property} of $(d+1)$-uniform hypergraphs is a class  $P\subseteq G^{d+1}$ of
such hypergraphs closed by isomorphism; here $G^{d+1}$ denotes the class of all $(d+1)$-uniform hypergraphs. 
In spite of this formalization, we use logical symbols such as $\land$ and $\lor$ instead of the set theoretical counterparts $\cap$ and $\cup$.
Thus, for instance, we write $\lnot P$ for $G^{d+1}\smallsetminus P$. We also prefer $H\models P$ to $H\in P$.

Each property $P$ gives rise to a sequence $\mathbb{P}_n[P]$ with
$\mathbb{P}_n[P]=\mathbb{P}[G^{d+1}(n,p)\models P]$;
 it is the asymptotic behavior  of this sequence we shall be interested in.
A property $P\subseteq G^{d+1}$ is said to hold \emph{asymptoticaly almost surely} (or simply \emph{almost surely}) if $\mathbb{P}_n[P]\to 1$. In this case we say simply that \emph{$P$ holds a.a.s.}. A property $P$ is said to hold \emph{almost never} if its negation $\lnot P$ holds almost surely. Very often, it is the case that a property $P$
turns from holding almost never to holding almost surely in a narrow range of the edge probability $p$.

\begin{Def}

We say a function $\tilde{p}:\mathbb{N}\to[0,1]$ is an \emph{increasing local threshold} for the property $P$ if there are functions $p_1,p_2$ such that
 $p_1\ll\tilde{p}\ll p_2$ and, for all $p:\mathbb{N}\to[0,1]$ satisfying $p_1\ll p\ll p_2$, the following conditions hold:

 \begin{enumerate}

   \item   If $p\gg\tilde{p}$ then $P$ holds a.a.s. in $G^{d+1}(n,p)$.

  \item If $p\ll\tilde{p}$ then $\lnot P$ holds a.a.s. in $G^{d+1}(n,p)$.

         \end{enumerate}
The function $\tilde{p}$ is a \emph{decreasing local threshold} for $P$ if it is an increasing local threshold for $\lnot P$. 
Finally, $\tilde{p}$ is  a \emph{local threshold} for $P$ if it is either an increasing or a decreasing local threshold.

\end{Def} 

Above and in all that follows, for eventually positive functions $f,g$, both  expressions $f\ll g$ and $f=o(g)$ mean $\lim\frac{f(n)}{g(n)}=0$.
Also, $f\sim g$ means $\lim\frac{f(n)}{g(n)}=1$.

Notice that if $P$ is monotone increasing and has a threshold $\tilde{p}$ in the usual sense \cite{erdos}, then $\tilde{p}$ is an increasing local threshold for $P$.
Furthermore, $\tilde{p}$ is the only local threshold for $P$. On the other hand, we shall see examples of non-monotone properties with more than one local threshold.

As far as thresholds are concerned, the following proposition is a generalization of a classical result of Erd\H os, R\'enyi and Bollob\'as, stated and proved by Vantsyan in \cite{vants}.

\begin{Thrm}       \label{vantsyan}

Fix a finite $(d+1)$-uniform hypergraph $H$ and let 

       $$\rho:=\max\left\{\frac{|E(\tilde{H})|}{|V(\tilde{H})|}\mid\tilde{H}\subseteq H, E(\tilde{H})>0\right\}.$$

Then the function $p(n)=n^{-\frac{1}{\rho}}$ is a threshold for the property of containment of $H$ as a sub-hypergraph.

\end{Thrm}

That is to say, $n^{-\frac{1}{\rho}}$ is a threshold for the appearance of small sub-hypergraphs with maximal \emph{density} $\rho$. We will need this later.

 \subsection{Logarithmico-Exponential Functions}


An obstruction to getting interesting results concerning the asymptotic behavior of the sequence $\mathbb{P}_n[P]$, specially regarding its convergence,
is the possibility that the function $p$ could oscillate infinitely often between two different values, so that the corresponding probabilities also do. This would rule out convergence for trivial
reasons and calls for a restriction on the class of the $p:\mathbb{N}\to[0,1]$ considered.

One possibility is to take only the $p$'s in a class entirely made of eventually monotone functions. A natural choice is Hardy's class of \emph{logarithmico-exponential functions},
or \emph{$L$-functions} for short, as presented in \cite{hardy}. In a nutshell, this consists of
 the eventually defined real-valued functions defined by a finite combination of the ordinary algebraic symbols and the functional symbols $\log(\ldots)$ and $\exp(\ldots)$.
on the variable $n$ (we require that, in all ``stages of construction", the functions take only real values).
By induction on the complexity of $L$-functions, one can easily show that this class meets our requirement and even more. 
\begin{Thrm}

Any L-function is eventually continuous, of constant
sign, and monotonic, and, as $n\to +\infty$, either converges to a real number or tends to $\pm\infty$. In particular, if $f$ and $g$ are eventually positive $L$-functions, exactly one of the following relations holds.

\begin{enumerate}

\item $f\ll g$
\item$f\gg g$
\item $f\sim c\cdot g$, for some constant $c\in\mathbb{R}.$

\end{enumerate}
 
\end{Thrm}

            \subsection{First Order Logic of Hypergraphs}

Having narrowed the class of possible edge probability functions, we now turn to a similar procedure on the class of properties of hypergraphs.

The \emph{first order logic of $(d+1)$-uniform hypergraphs} $\mathcal{FO}$ is the relational logic with language $\{\sigma\}$, where $\sigma$ is a $(d+1)$-ary predicate.
The semantics of $\mathcal{FO}$ is given by quantification over vertices and giving the formula $\sigma(x_0,x_1,\ldots,x_d)$ the interpretation ``$\{x_0,x_1,\ldots,x_d\}$ is an edge". Thus, given a hypergraph $H$ and $\phi\in\mathcal{FO}$, we write $H\models\phi$ if $\phi$ holds in $H$.

Each $\phi\in\mathcal{FO}$ \emph{represents} the property $P=\{H\in G^{d+1} \mid H\models\phi\}$ so that $H\models P$ if and only if $H\models\phi$.
We say a property $P$ of $(d+1)$-uniform hypergraphs is \emph{elementary} if it can be represented by a formula in $\mathcal{FO}$. In this case, we abuse notation by writing
$P\in\mathcal{FO}$ and, when no possibility of confusion arises, make no distinction between $P$ and the first order formula defining it.

A \emph{first order (or elementary) theory} is simply a subclass $\Theta\subseteq\mathcal{FO}$, or, still abusing notation, a class of elementary properties. 
We write $H\models\Theta$ if $H\models\phi$ for all $\phi\in\Theta$.
We say a property $P$ is \emph{axiomatizable} if there is a first order 
  theory $\Theta$ such that, for all hypergraphs $H$, one has $H\models P$ if and only if $H\models\Theta$. If there is such a finite $\Theta=\{\phi_1,\ldots,\phi_k\}$, then
  $P$ is equivalent to $\phi\land\ldots\land\phi_k$, and is, therefore, elementary. 
  
  It is easy to see that the property $P_k$ of having at least $k$ vertices is elementary. If $\phi_k$ is a elementary
  sentence meaning $P_k$, then the property $\textbf{INF}$ of being infinite is axiomatizable by the theory 
  $\{\phi_k\}_{k\geq 1}$. The next proposition can be used to show that the property $\textbf{FIN}$ of being finite is not axiomatizable and that $\textbf{INF}$ is not elementary.

Fix a first order theory $\Theta\subseteq\mathcal{FO}$ and a property $P\in\mathcal{FO}$. We say \emph{$P$ is a semantic consequence of $\Theta$},
 and write $\Theta\models P$, if $P$ is satisfied in all hypergraphs that satisfy all
of $\Theta$, that is to say, if $H\models\Theta$ implies $H\models P$. One can define a deductive system in which all derivations are finite sequences of formulae in $\mathcal{FO}$, giving rise to the concept of 
$P$ being a \emph{syntatic consequence of $\Theta$}, meaning that some $\mathcal{FO}$-formula defining $P$ is the last term of a derivation that only uses formulae in $\Theta$ as axioms. 

We shall need a few basic results in first-order logic.
One piece of information in G\"odel's Completeness Theorem is the fact that one can pick such a deductive system in a suitable fashion so as to make the concepts of semantic and syntatic consequences identical. As derivations are finite sequences of formulae, the following \emph{Compactness Result} follows.

\begin{Prop}

If $P$ is a semantic consequence of $\Theta$ then it is a semantic consequence of a finite subclass of $\Theta$.
In particular, if every finite subclass of $\Theta$ is consistent then $\Theta$ is consistent.

\end{Prop}

The last part comes from substitution of $P$ by any contradictory property. A careful analysis of the argument on the proof of G\"odel's Completeness Theorem shows the \emph{Downward L\"owenhein-Skolem Theorem},
that if $\Theta$ is a consistent theory (that is, satisfied by some hypergraph) then there is a hypergraph on a finite or countable number of vertices satisfying $\Theta$. 

It is convenient to
suppose that first order theories are closed under semantic implication and we will always assume 
that. So, if $\tilde{\Theta}$ is any elementary theory and $\tilde{\Theta}\models\phi$, then $\phi\in\tilde{\Theta}$.
From now on, we will call a first order theory simply a theory.

        \subsection{Zero-One Laws and Complete Theories}

The above observations will be useful in obtaining the following convergence results involving all properties in $\mathcal{FO}$.

\begin{Def}

We say a function $p: \mathbb{N}\to [0,1]$ is a \emph{zero-one law} if, for all $P\in\mathcal{FO}$, one has

      $$\lim_{n\to\infty}\mathbb{P}_n[P]\in\{0,1\}.$$

\end{Def}

Above we mean that for every $P\in\mathcal{FO}$ the limit exists and is either zero or one. 

There is a close connection between zero-one laws and the concept of completeness. We say a theory $\Theta$ is \emph{complete} if, for
every $P\in\mathcal{FO}$, exactly one of $\Theta\models P$ or $\Theta\models\lnot P$ holds. 

Given $p:\mathbb{N}\to[0,1]$, the \emph{almost sure theory} of $p$ is defined by

         $$\Theta_p:=\{P\in\mathcal{FO} \mid \mathbb{P}_n[P]\to 1\}.$$
So $\Theta_p$ is the class of elementary properties of $G^{d+1}(n,p)$ that hold almost surely. 
Clearly, if $\phi_1,\ldots,\phi_k\in\Theta_p$ and $\{\phi_1,\ldots,\phi_k\}\models\phi$, then $\phi\in\Theta_p$.
By compactness, $\Theta_p$ is closed under semantic consequence: $\Theta_p\models\phi$ if and only if $\phi\in\Theta_p$.
In particular, since a contradiction never holds, $\Theta_p$ is consistent. Moreover as, for every $m\in\mathbb{N}$, the property of having at least 
$m$ vertices is elementary and holds almost surely, $\Theta_p$ has no finite models.

The connection between completeness and zero-one laws is given by the next theorem.

\begin{Thrm}

The function $p$ is a zero-one law if, and only if, $\Theta_p$ is complete.

\end{Thrm}

\begin{proof}

Suppose $\Theta_p$ is complete and fix $P\in\mathcal{FO}$. As $\Theta_p$ is complete, either $P$ or $\lnot P$ is a semantic consequence of $\Theta_p$.
 If $\Theta_p\models P$, by compactness, there is a finite set
$\{P_1,P_2\ldots,P_k\}\subseteq\Theta_p$ such that $\{P_1,P_2\ldots,P_k\}\models P$. Therefore
$\mathbb{P}_n[P_1\land P_2\cdots\land P_k]\leq\mathbb{P}_n[P].$

As $\mathbb{P}_n[P_1\land P_2\cdots\land P_k]\to 1$ we have also $\mathbb{P}_n[P]\to 1$. Similarly, if $\Theta_p\models\lnot P$ one has $\mathbb{P}_n[\lnot P]\to 1$, so that $\mathbb{P}_n[P]\to 0$. As $P$ was arbitrary, 
$p$ is a zero-one law.

Conversely, if $p$ is a zero-one law then, for any $P\in\mathcal{FO}$, we have either $P\in\Theta_p$ or $\lnot P\in\Theta_p$. One cannot have both, as $\Theta_p$ is consistent. So $\Theta_p$ is complete.
\end{proof}

As $\Theta_p$ is consistent and has no finite models, G\"odel's Completeness Theorem and L\"owenhein-Skolem give that the requirement of $\Theta_p$ being complete is equivalent to asking that all
countable models of $\Theta_p$ satisfy exactly the same first-order properties, a situation described in Logic by saying that all countable models are \emph{elementarily equivalent}. One obvious sufficient condition is that $\Theta_p$ be \emph{$\aleph_0$-categorical}, that is, that $\Theta_p$ has, apart from isomorphism, a unique countable model. We shall see several examples where $\Theta_p$ is $\aleph_0$-categorical and other examples where the countable models of $\Theta_p$ are elementarily equivalent but not necessarily isomorphic.

We summarize the above observations in the following corollary, more suitable for our applications.

\begin{Cor}

A function $p$ is a zero-one law if, and only if, all models of the almost sure theory $\Theta_p$ are elementarily equivalent. 
In particular, if $\Theta_p$ is $\aleph_0$-categorical, then $p$ is a zero-one law.

\end{Cor}

Uses of the above result require the ability to recognize when any two models $H_1$ and $H_2$  of $\Theta_p$ are elementarily
equivalent. This is, usually, a simple matter in case $H_1$ and $H_2$ are isomorphic.
It is convenient to have at hand an instrument suitable to detecting when two structures of a first-order theory are elementarily equivalent regardless of being isomorphic.

Next, we briefly present the Ehrenfeucht-Fra\"iss\'e Game, which is a classic example of such an instrument.

           \subsection{The Ehrenfeucht-Fra\"iss\'e Game} \label{game}

This game has two players, called Spoiler and Duplicator, and two uniform hypergraphs $H_1$ and $H_2$ conventionally on disjoint sets of vertices. These
hypergraphs are known to both players. The game has a certain number $k$ of rounds, fixed in advance and also known to both players.
In each round, Spoiler selects one vertex not previously selected in either hypergraph and then Duplicator selects another vertex not previously selected in the other hypergraph.
At the end of the $k$-th round, the vertices $x_1,\ldots,x_k$ have been chosen on $H_1$ and $y_1,\ldots,y_k$ on $H_2$. Duplicator wins if, for all
$\{i_0,i_1,\ldots,i_d\}\subseteq\{1,2\ldots,k\}$, $\{x_{i_0},\ldots,x_{i_d}\}$ is an edge in $H_1$ if and only if the corresponding $\{y_{i_0},\ldots,y_{i_d}\}$ is an edge in $H_2$.
Spoiler wins if Duplicator does not. We denote the above described game by $\ehf\left(H_1,H_2;k\right)$.

As a technical point, the above description of the game works only if $k\leq\min\{\left|H_1\right|,\left|H_2\right|\}$. If that is not the case, we adopt the convention that Duplicator 
wins the game if, and only if, $H_1$ and $H_2$ are isomorphic.

This kind of reasoning was used for the first time by R. Fra\"iss\'e in his PhD thesis in the more general context of purely relational structures with finite predicate symbols. 
The formulation as a game is due to Andrzej Ehrenfeucht. A proof of the following proposition in the particular case of
graphs can be found in Joel Spencer's book  \cite{spencer}, whose argument applies, \emph{mutatis mutandis}, to uniform hypergraphs.

\begin{Prop}   \label{winning strategy}

A necessary and sufficient condition for the hypergraphs $H_1$ and $H_2$ to be elementarily equivalent is that, for all $k\in\mathbb{N}$, Duplicator has a winning strategy for the game $\ehf(H_1,H_2;k)$.

\end{Prop}

Now it is easy to see the connection of the game to zero-one laws.

\begin{Cor}

If for all countable models $H_1$ and $H_2$ of $\Theta_p$ and all $k\in\mathbb{N}$ Duplicator has a winning strategy for $\ehf(H_1,H_2;k)$ then $p$ is a zero-one law.

\end{Cor}

For the description of some situations when there is a winning strategy for Duplicator without $H_1$ and $H_2$ being necessarily isomorphic, we refer the reader to
      \cite{tese}, where statements can be found in the general case of uniform hypergraphs. We also refer to \cite{spencer} for the proofs of these statements in the particular case of graphs.
  The generalization of the arguments to fit uniform hypergraphs are straightforward.

     \section{The Space of Completions} \label{completions}

For the convenience of the exposition, we construct a topological space. Later, examples will be given that establish a relation
between what follows and random hypergraphs. This space is always compact, metrizable and totally disconnected, but its structure (and that of 
 the set of its limit points, the derived set) varies. For instance, the space may be finite, countably infinite or a Cantor set.

Let $\Theta$ be an elementary theory on $\mathcal{FO}$.
We say a theory $\tilde{\Theta}\supseteq\Theta$ is a \emph{completion} of $\Theta$ if it is a complete theory, 
which is to say: $\tilde{\Theta}$ is consistent and, for all first-order sentences $\phi$, either
$\phi\in\tilde{\Theta}$ or $\lnot\phi\in\tilde{\Theta}$.    

 Define
 \begin{gather*}
          \mathcal{K}(\Theta)=\{\tilde{\Theta}\supseteq\Theta \ |\ \tilde{\Theta} \text{ is a completion of } \Theta\}, \\
      A_\phi=\{\tilde{\Theta}\in\mathcal{K}(\Theta)\ |\ \phi\in\tilde{\Theta}\}, \quad      \mathcal{A}=\{A_\phi\mid\phi\in\mathcal{FO}\}.
 \end{gather*}            
       The relations $A_\phi\cap A_\psi=A_{\phi\land\psi}$, ${(A_\phi)}^\complement=A_{\lnot\phi} $, and $ A_{\phi\lor\lnot\phi}=\mathcal{K}(\Theta)$
show that $\mathcal{A}$ is an algebra of subsets of $\mathcal{K}(\Theta)$. Furthermore, $\mathcal{A}$ is a basis for a topology in $\mathcal{K}(\Theta)$ in relation to which the sets $A_\phi$
 are all clopen (that is, closed and open).
Clearly, $\Theta$ is inconsistent if and only if $\mathcal{K}(\Theta)=\emptyset$ and $\Theta$ is complete if
and only if $\mathcal{K}(\Theta)=\{\Theta\}$.

An alternative construction could be given by considering $\mathcal{K}(\Theta)$ as a class of \emph{structures} instead of a class of completions.
Consider the class $\mathcal{G}$ of all $(d+1)$-uniform hypergraphs. Let
         $\mathcal{K}$ be the quotient space $\mathcal{G}/{\sim}$, where $\sim$ is the elementary equivalence relation.
           We then turn $\mathcal{K}$ into a topological space by considering the sets of the form
       $A_\phi=\{[H]\in\mathcal{K}\ |\ H\models\phi\}$ to be
basic open sets. Finally, set 
                  $$\mathcal{K}(\Theta)=\{[H]\in\mathcal{K}\ |\ H\models\Theta\}\subseteq\mathcal{K}.$$ 
Although these two constructions impose different meanings on $A_\phi$ and $\mathcal{K}(\Theta)$, the correspondence between them is natural. Therefore we will
   use the same notation in either context.

The following structure arises naturally when one considers the space $\mathcal{K}(\Theta)$.
In what follows, $\mathcal{T}$ is a rooted tree and recall that a \emph{branch} is a maximal 
totally ordered subset.

    \begin{Def}
    
$\mathcal{T}$ is a spanning tree for $\mathcal{K}(\Theta)$ if the following conditions hold:

         \begin{itemize}
        
                   \item Every node of $\mathcal{T}$ is a pair $(A,h)$, where $A$ is a non-empty clopen set and $h\in\mathbb{N}$ is the height of the node.
                   \item $\mathcal{T}$ is rooted on $(\mathcal{K}(\Theta),0)$.
                   \item For every node $(A,h)$, its children are $(A_j,h+1), 1\leq j\leq k$, where the sets $A_j$ are disjoint and $A=\bigcup_{1\leq j\leq k} A_j$.
                   \item For every branch $\mathcal{B}$, $\bigcap_{(A,h)\in\mathcal{B}}A=\{\Theta_\mathcal{B}\}$ is a singleton.

          \end{itemize}

   \end{Def}    
   
   The option of putting every node of a spanning tree to be a pair $(A_\phi,h)$ is a technical convenience, designed to allow a clopen set $A_\phi$ to appear in multiple
    nodes of $\mathcal{T}$.  Many times we abuse notation by referring 
    to the node $A_\phi$ when we really mean $(A_\phi,h)$.
   
   Note that the compactness of $\mathcal{K}(\Theta)$ implies that every node has a finite number of children and that the set $\mathcal{N}$ of nodes of $\mathcal{T}$ 
   forms a basis for the topology of $\mathcal{K}(\Theta)$, and that, given any two nodes $A_{\phi_1},A_{\phi_2}\in\mathcal{N}$, one of the following three conditions holds:
   $A_{\phi_1}\subseteq A_{\phi_2}$, $A_{\phi_2}\subseteq A_{\phi_1}$, $A_{\phi_1}\cap A_{\phi_2}=\emptyset$.
   Note also that, if $\mathcal{B}$ is a branch, then $\Theta_\mathcal{B}$ is a complete theory and
     $\mathcal{K}(\Theta)=\{\Theta_\mathcal{B} \ | \ \mathcal{B} \text{ is a branch of } \mathcal{T}\}$.
   We shall now see that $\mathcal{K}(\Theta)$ always admits a spanning tree and, along the way, prove some of its basic properties.
      

    \begin{figure}[!htb]
    \centering
    \includegraphics[height=150mm,angle=0]{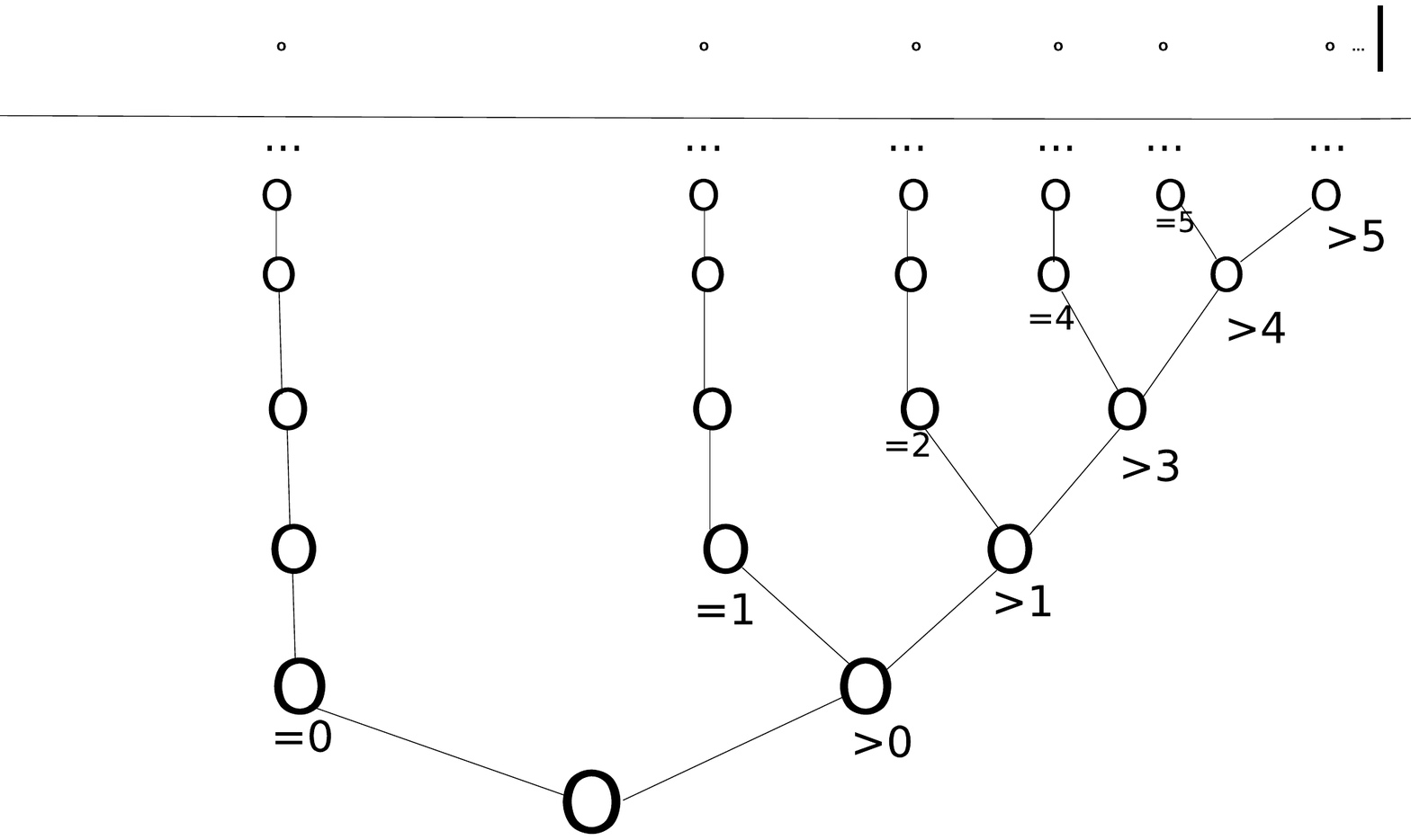}
    \caption{($d=2$) A spanning tree}
    \label{figRotulo}
  \end{figure}

\begin{Prop}

$\mathcal{K}(\Theta)$ is a totally disconnected second-countable compact metrizable space and admits a spanning tree.

\end{Prop}

\begin{proof}

It suffices to consider $\Theta=\emptyset$ (that is, the particular case of $\mathcal{K}=\mathcal{K}(\emptyset)$), 
because $\mathcal{K}(\Theta)$ is a closed subset of $\mathcal{K}$.
Second countability is immediate because the basis exhibited above is countable.

To see that $\mathcal{K}$ is totally disconnected, it suffices to see that it is totally separated.
To this end, take $\Theta_1,\Theta_2\in\mathcal{K}$, 
such that $\Theta_1\neq\Theta_2$. 
There is $\phi$ such that $\Theta_1\models\phi$ but $\Theta_2\models\lnot\phi$. Then
$\Theta_1\in A_\phi$, $\Theta_2\in A_{\lnot\phi}$, $A_\phi\cap A_{\lnot\phi}=\emptyset$ 
and $A_\phi\cup A_{\lnot\phi}=\mathcal{K}$ .

 For compactness, it clearly suffices to see that any cover of $\mathcal{K}$ by basic open sets has
 a finite subcover. With this in mind, note that $\mathcal{K}=\bigcup_{\phi\in I}A_\phi$ if, and only
 if, the theory $\{\lnot\phi\ |\phi\in I\}$ is inconsistent. In this case, by the compactness of first
 order logic, there is an inconsistent finite sub-theory $\{\lnot\phi_1,\lnot\phi_2,\ldots,\lnot\phi_k\}$. Hence
 $A_{\phi_1}\cup A_{\phi_2}\cup\ldots\cup A_{\phi_k}$ is a finite subcover of
 $\bigcup_{\phi\in I}A_\phi$ and we are done.

 One can give $\mathcal{K}$ an explicit metric (actually, an ultrametric). Define the size of an elementary sentence $\phi$ to be the number of characters in $\phi$.
Set the distance between two distinct complete theories $\Theta_1,\Theta_2\in\mathcal{K}$
 to be $2^{-N}$, where $N$ is the smallest size of a sentence in the symmetric difference $\Theta_1\Delta \Theta_2$.

 Notice that balls of radius $2^{-N}$ form a partition of $\mathcal{K}(\Theta)$ into finitely many clopen sets. These form, by definition, the $N$-th level of a spanning tree.
 \end{proof}

Note that all branches are infinite.
 Isolated points in $\mathcal{K}(\Theta)$ correspond to \emph{isolated branches}, that is, branches for which the first coordinate is eventually constant.
 Therefore, if $\mathcal{K}(\Theta)$ admits a spanning tree none of whose branches is isolated, then
it is a Cantor space.

Also, the algebra $\mathcal{A}$ consists of the sets $A$ which are finite disjoint unions
          $$A=\bigsqcup_{1\leq i\leq k}A_{\phi_i}, \qquad A_{\phi_i}\in\mathcal{N}.  $$

                \subsection{Borelian Probabilities on $\mathcal{K}(\Theta)$}

In what follows, we describe how a spanning tree can be used to specify borelian probability measures on the space $\mathcal{K}(\Theta)$.      
  
Let $\mu$ be a borelian probability measure on $\mathcal{K}(\Theta)$ and $\mathcal{N}$ the set of all nodes on a spanning tree $\mathcal{T}$.
Then the restriction $\mu\restriction\mathcal{N}$ is \emph{hereditarily consistent}, in the sense that $\mu(\mathcal{K}(\Theta))=1$ and, for every $A\in\mathcal{N}$,
                                        $$\mu(A)=\sum_{B\text{ is a child of } A}\mu(B).$$ 
    
    Conversely, we have the next lemma.

 \begin{Lemma} \label{weightedspanningtree}
 
Any hereditarily consistent $\mu_\mathcal{N}:\mathcal{N}\to[0,1]$ extends uniquely to a borelian probability measure $\mu$ on $\mathcal{K}(\Theta)$.
 
 \end{Lemma}

 \begin{proof}
 
 An easy induction shows that $\mu_\mathcal{N}$ is finitely additive.
 It is then straightforward to see that 
 $\mu_\mathcal{N}$ extends uniquely to a finitely additive function $\mu_\mathcal{A}:\mathcal{A}\to[0,1]$.
As the $A\in\mathcal{A}$ are clopen and $\mathcal{K}(\Theta)$ is compact, we see that $\mu_\mathcal{A}$ is
in fact $\sigma$-additive. Therefore, the usual Hahn-Kolmogorov extension Theorem gives that $\mu_\mathcal{A}$
extends uniquely to a borelian probability measure $\mu$ on $\mathcal{K}(\Theta)$. 
  \end{proof}        
 
 Hence we have a correspondence between hereditarily consistent functions on $\mathcal{N}$ and borelian probability measures
 on $\mathcal{K}(\Theta)$. 
 We call a pair $(\mathcal{T},\mu)$, where $\mathcal{T}$ is a spanning tree and $\mu$ is a hereditarily consistent function on the set
 of nodes of $\mathcal{T}$, a \emph{weighted spanning tree}.

            \subsection{Application to Random Hypergraphs}

We are interested in the case $\Theta=\Theta_p$, where $p:\mathbb{N}\to[0,1]$ is the edge probability of the random hypergraph
$G^{d+1}(n,p)$. 

If $p$ is a zero-one law, then $\Theta_p$ is complete and $\mathcal{K}(\Theta_p)$ trivializes to a one-point space.

Sometimes a zero-one law is too much to ask and it is interesting to consider a related broader concept.

\begin{Def}

A function $p:\mathbb{N}\to[0,1]$ is a \emph{convergence law} if, for all $\phi\in\mathcal{FO}$, the sequence $(a_n)$, with
 $a_n=\mathbb{P}_n[A_\phi]$, converges.

\end{Def}

When $p$ is a convergence law, it induces a borelian probability measure on $\mathcal{K}(\Theta_p)$.
In order to see that a certain edge probability $p$ is a convergence law, it suffices to see that the probabilities
of the nodes on a spanning tree for $\mathcal{K}(\Theta_p)$ converge.
In what follows, $\mathcal{T}$ is a spanning tree on $\mathcal{K}(\Theta_p)$
with set of nodes $\mathcal{N}$.

\begin{Prop}

If, for every $A\in\mathcal{N}$, the sequence $\mathbb{P}_n[A]$ converges to the real number $\mu(A)$, 
then $p$ is a convergence law, $(\mathcal{T},\mu)$ is a weighted spanning tree on $\mathcal{K}(\Theta_p)$ and
 $\mu$ extends uniquely to a probability measure on $\mathcal{K}(\Theta_p)$ such that, for every elementary $\phi$,
                      $$\mu(\phi)=\lim_{n\to\infty}\mathbb{P}_n[A_\phi].$$

\end{Prop}

\begin{proof}

To see that $p$ is a convergence law, fix an elementary sentence $\phi$. 
Then $\mathbb{P}[\phi]$ converges to $\sum_i\mu(A_{\phi_i})$, where
$A_\phi=\bigsqcup_{1\leq i\leq k}A_{\phi_i}, A_{\phi_i}\in\mathcal{N}.  $
The remaining claims are easy, in view of Lemma \ref{weightedspanningtree}.
\end{proof}

                                              \section{Big-Bang}

                         \subsection{Counting of Berge-Tree Components}

Now we proceed to investigate zero-one and convergence laws in the early stages of the evolution of $G^{d+1}(n,p)$. More precisely, we investigate edge functions before the double jump:
            $$0\leq p(n)\ll n^{-d}.$$
For functions $p$ in that range, $G^{d+1}(n,p)$ almost surely has no cycles, in the following sense:

 Recall that the \emph{incidence graph} $G(H)$ of a hypergraph $H$ is a bipartite graph with $V(H)$ on one side and $E(H)$ on the other and such that, for all $v\in V(H)$ and $e\in E(H)$, there is an edge connecting
$v$ and $e$ in $G(H)$ if, and only if, $v\in e$ in $H$. A \emph{cycle} in a hypergraph is a cycle in its incidence graph. We say a hypergraph is \emph{Berge-acyclic} if has no cycles. From now on, we shall refer to Berge-acyclic hypergraphs simply as \emph{acyclic hypergraphs}.

\begin{Def}

A \emph{Berge-tree} is a connected acyclic uniform hypergraph. The \emph{order} of a finite Berge-tree is its number of edges.

\end{Def}

  \begin{figure}[!htb]
    \centering
    \includegraphics[height=30mm,angle=90]{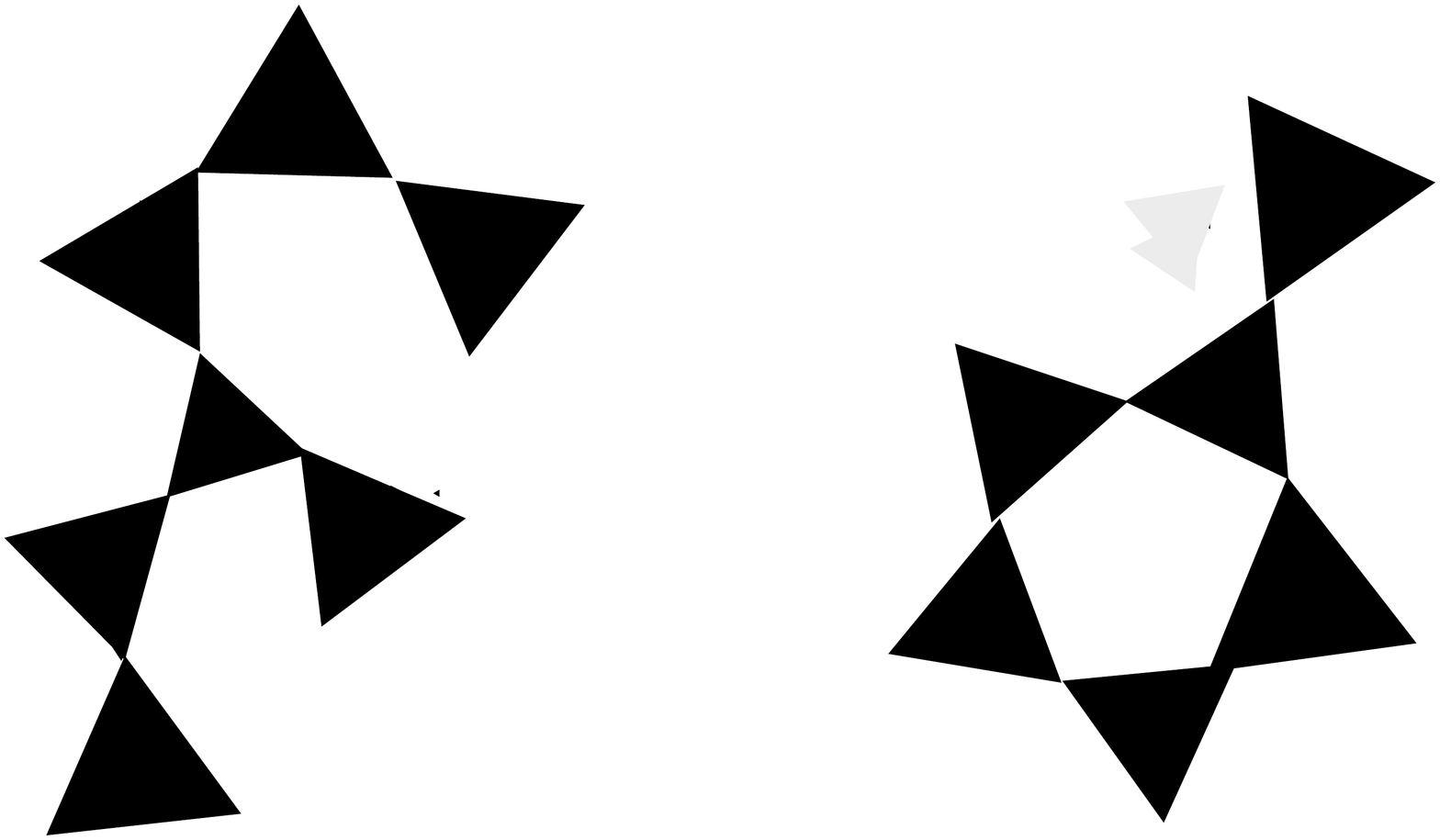}
    \caption{($d=2$) Left: a Berge-tree. Right: Not a Berge-tree.}
    \label{figRotulo}
  \end{figure}

\begin{Prop}

If $0\leq p\ll n^{-d}$ and $C$ is a fixed finite cycle, then a.a.s. $G^{d+1}(n,p)$ does not have a copy of 
$C$ as a sub-hypergraph.

\end{Prop}


       

This proposition is a direct corollary of the following. On the other hand, note that the converse implication is not obvious.

\begin{Prop}

 If $p(n)=(n^d\cdot h(n))^{-1}$, where $h(n)\to\infty$,
  then a.a.s. $G^{d+1}(n,p)$ is acyclic.

\end{Prop}

\begin{proof}

We give a proof for the particular case $d=1$ of graphs. The argument for general $d$ is similar, with more cumbersome notation. 
Let $C_t$ be the number of cycles of size $t$ and $C_{x,y}=\sum_{x<t\leq y}C_t$ be the number of cycles of size greater than $x$ and at most $y$. We prove that
 $\mathbb{E}[C_{2,\infty}]=o(1)$.
 Indeed, we have
        $$\mathbb{E}[C_t]=\frac{1}{2}\cdot{n\choose t} p^t (t-1)!\leq(h(n)^t\cdot t)^{-1},$$ a decreasing function of $t$.
        Hence, summing over $t$, with $l<t\leq 2l$,
       $$\mathbb{E}[C_{l,2l}]\leq\frac{l}{\left(h(n)\right)^l\cdot l}=h(n)^{-l}$$        
   and, therefore, comparing with a geometric progression, 
       $$\mathbb{E}[C_{2,\infty}]=\sum_{k=1}^{\infty}\mathbb{E}\left[C_{2^{k},2^{k+1}}\right]
       =\sum_{k=1}^{\infty}\left(h(n)\right)^{-2^k}=\frac{h(n)^{-2}}{1-h(n)^{-2}}=o(1).$$

\end{proof}

In view of the above, as far as all our present discussions are concerned, the hypergraphs we deal with are disjoint unions of Berge-trees. Getting more precise information on the statistics of the number of connected components that are finite Berge-trees of a given order is the most important piece of information to getting zero-one and convergence laws for $p\ll n^{-d}$.

To this end, we define the following random variables. Below $\tau\in\textbf{T}$, where $\textbf{T}$ is the set of all isomorphism classes of Berge-trees of order $l$ on $v=1+ld$ labelled vertices.

\begin{Def}

 $A^\tau(l)$ is the number of finite Berge-trees of order $l$ and isomorphism class $\tau$ in $G^{d+1}(n,p)$.

\end{Def}

A Berge-tree of order $l$ is, in particular, a hypergraph on $v=1+ld$ vertices. Let $c^\tau(l)$ be the number of Berge-trees of order $l$ and isomorphism class $\tau$ on $v=1+ld$ labelled vertices. To each set $S$ of $v$ vertices on $G^{d+1}(n,p)$ there corresponds 
 the collection of indicator random variables $X_S^1,X_S^2,\ldots,X_S^{c^\tau(l)}$, each indicating that one of the potential $c^\tau(l)$ Berge-trees of order $l$ and isomorphism class $\tau$ in $S$ is present and is a component. Therefore one has

             $$A^{\tau}(l)=\sum_{S,i}X_S^i,$$
where $S$ ranges over all $v$-sets and $i$ ranges over $\{1,2,\ldots,c^{\tau}(l)\}$.

Note that a connected component isomorphic to a Berge-tree of a certain isomorphism class is, in particular, an induced copy of that Berge-tree.
Next we show that a local threshold for containment of a Berge-tree of given order as a connected component is the same for containing Berge-trees of that order as sub-hypergraphs, not necessarily induced.

In the next proposition, the reader may find the condition 
$$p\leq C(\log n)n^{-d}$$
  in $2$ rather strange, since it mentions functions outside the scope $p\ll n^{-d}$ of the present section. The option to putting this more general proposition here 
reflects the convenience that it has exactly the same proof and that the full condition will be used later.

\begin{Prop} \label{threshold}

Set $v=1+ld$. The function $n^{-\frac{v}{l}}$ is a local threshold for containment of Berge-trees of order $l$ as components. More precisely:

      \begin{enumerate}

   \item If $p\ll n^{-\frac{v}{l}}$ then a.a.s. $G^{d+1}(n,p)$ has no Berge-trees of 
order $l$ as sub-hypergraphs. 

   \item If $n^{-\frac{v}{l}}\ll p\leq C(\log n)n^{-d}$ where $C<\frac{d!}{1+ld}$ then, for any $k\in\mathbb{N}$ and $\tau\in\textbf{T}$, a.a.s. $G^{d+1}(n,p)$ has at least $k$ connected components isomorphic the Berge-tree of order $l$ and isomorphism class $\tau$.

       \end{enumerate}

\end{Prop}

We refer the reader to \cite{tese} for the proof of the above proposition, which is a direct application of the first and second moment methods.
The next proposition proves parts $(i) (a)$ and $(i) (b)$ of Theorem \ref{grandeteorema}.

\begin{Prop}

If  $0\le p\ll n^{-(d+1)}$ or there is $l\in\mathbb{N}$ such that      $n^{-\frac{1+ld}{l}}\ll p\ll n^{-\frac{1+(l+1)d}{l+1}}$ then p is a zero-one law.

\end{Prop}

\begin{proof}

Note that, by the above proposition, all the countable models of the almost sure theory $\Theta_p$ are isomorphic, i.e., 
$\Theta_p$ is $\aleph_0$-categorical.
\end{proof}

Hence, for the edge probabilities considered here, $\mathcal{K}(\Theta_p)$ is simply the one-point space $\{\Theta_p\}$.

Let $\Theta_l$ be the first order theory consisting of a scheme of axioms excluding the existence of
cycles and Berge-trees of order $\ge l+1$ and a scheme that assures the existence of infinite copies
of each type of Berge-trees of order $\le l$. Then $\Theta_l$ is an axiomatization for $\Theta_p$, where $p$ is
as above.

         \subsection{Just Before the Double Jump}

Consider now an edge function $p$ such that  for all $\epsilon >0$, $n^{-(d+\epsilon)}\ll p\ll n^{-d}$. Such functions would include, for instance, $p(n)=(\log n)^{-1}n^{-d}$.
The countable models of the almost sure theories for such $p$'s must be acyclic and have infinite components isomorphic to Berge-trees of all orders. But in this range a
new possibility can occur: the existence of components that are Berge-trees of infinite order. There may or there may not be such components, and therefore
the countable models of $\Theta_p$ are not $\aleph_0$-categorical.

These infinite components do not matter from a first-order perspective, as they will be  ``simulated" by sufficiently large finite components.
Because first-order properties are represented by finite formulae, with finitely many quantifications, this will establish that all models of $\Theta_p$ are elementarily equivalent
in spite of not being $\aleph_0$-categorical.

     \subsubsection{Rooted Berge-Trees}  \label{rooted Berge-trees}

The following two results are stated and proved in Spencer's book \cite{spencer} for trees, which are
particular cases of Berge-trees when $d=1$. The situation is similar to that of section \ref{game}: similar arguments apply to the other values of $d$.





\begin{Prop}

Let $H_1$ and $H_2$ be two acyclic graphs in which every finite Berge-tree occurs as a component an infinite number of times. Then $H_1$ and $H_2$ are elementarily equivalent.

\end{Prop}

It is convenient to emphasize that, above, $H_1$ and $H_2$ may or may not have infinite components.

The next proposition proves part $(i)(c)$
of \ref{grandeteorema}.

\begin{Prop}

Suppose $p$ is an edge function satisfying, for all $\epsilon>0$, 
                           $$n^{-(d+\epsilon)}\ll p\ll n^{-d}.$$
    Then $p$ is a zero-one law.

\end{Prop}

\begin{proof}

Consider an edge function $p$ such that $n^{-(d+\epsilon)}\ll p\ll n^{-d}$ for all $\epsilon>0$. All countable models of $\Theta_p$ satisfy the hypotheses of the above proposition. 
Therefore they are elementarily equivalent and these $p$'s are zero-one laws.
\end{proof}
So, here again, $\mathcal{K}(\Theta_p)=\{\Theta_p\}$.

Let $\Theta$ be the first order theory consisting of a scheme of axioms excluding the existence of cycles and a scheme that assures that every finite Berge-tree of any order appears as a component
an infinite number of times. Then $\Theta$ is an axiomatization for $\Theta_p$. 

          \subsection{On the Thresholds}

So far, we have seen that if $p$ satisfies one of the following conditions

\begin{enumerate}

   \item $0\le p\le n^{-(d+1)}$
   \item $n^{-\frac{1+ld}{l}}\ll p\ll n^{-\frac{1+(l+1)d}{l+1}}$, for some $l\in\mathbb{N}$
   \item $n^{-(d+\epsilon)}\ll p\ll n^{-d}$ for all $\epsilon>0$

\end{enumerate}
then $p$ is a zero-one law.

An $L$-function $p$ in the range $0\le p\ll n^{-d}$ that violates all the above three conditions must satisfy, for some $l\in\mathbb{N}$ and $c\in(0,+\infty)$, the condition
$p(n)\sim c\cdot n^{-\frac{1+ld}{l}}.$
In that case, $p$ is not a zero-one law. Our next goal is to show that those $p$'s are still convergence laws.

     \subsubsection{Limiting Probabilities on the Thresholds} 

Let $l\in\mathbb{N}$ and let $T_1,T_2,\ldots,T_u$ denote the collection of all possible (up to isomorphism) Berge-trees of order $l$ and, for a $u$-tuple $\textbf{m}=(m_1,\ldots,m_u)\in(\mathbb{N}_s)^u$, let $\sigma_{\textbf{m}}$ be the elementary property that there are precisely $m_i$ components $T_i$ if 
$0\leq i\leq s$ and at least $s+1$ components if $m_i=\mathcal{M}$.

The proposition that follows gives part $(ii)(a)$
of \ref{grandeteorema}.

\begin{Prop}

Let $p\sim c\cdot n^{-\frac{1+ld}{l}}$. The probabilities of all elements in the collection $\{\sigma_{\textbf{m}}\mid\textbf{m}\in(\mathbb{N}_s)^u, s\in\mathbb{N}\}$ converge as $n\to\infty$. Moreover, this collection
 is the set of nodes on a weighted spanning tree for $\mathcal{K}(\Theta_p)$, the weights being the asymptotic probabilities. In particular, $p$ is a convergence law.

\end{Prop}

\begin{proof}

It is clearly enough to consider the case $\textbf{m}\in(\mathbb{N})^u$.

The countable models of $\Theta_p\cup\{\sigma_{\textbf{m}}\}$ have no cycles, a countably infinite number of components of each Berge-tree of order $\le l-1$, no sub-hypergraph isomorphic to a Berge-tree of order $\ge l+1$ and exactly
$m_i$ components $T_i$ for each $i$. So $\Theta_p\cup\{\sigma_{\textbf{m}}\}$ is $\aleph_0$-categorical and, in particular, is complete. 

Tautologically no two of the $\sigma_{\textbf{m}}$ can hold simultaneously.

For each $i\in\{1,2,\ldots,u\}$, let $\tau_i$ be the isomorphism type of $T_i$. For notational convenience, set $c_i:=c^{\tau_i}(l)$ and $A_i:=A^{\tau_i}(l)$. 
The next lemma implies the remaining properties and, therefore, completes the proof.
\end{proof}

\begin{Lemma}  \label{poisson1}

In the conditions of the proof of the above proposition, the random variables $A_1,A_2,\ldots,A_u$ are asymptotically independent Poisson with means $\lambda_1=\frac{c_1}{v!}c^l,\lambda_2=\frac{c_2}{v!}c^l,\ldots,\lambda_u=\frac{c_u}{v!}c^l$. That is to say,

       $$p_{\textbf{m}}:=\lim_{n\to\infty}\mathbb{P}_n[\sigma_{\textbf{m}}]=\prod^u_{i=1}e^{-\lambda_i}\frac{\lambda_i^{m_i}}{m_i!}.$$
In particular  $$\sum_{\textbf{m}\in I}p_{\textbf{m}}=1.$$

\end{Lemma}

Again, we refer the reader to \cite{tese} for a proof.

So, for $p\sim c\cdot n^{-\frac{1+ld}{l}}$, the collection $\{\sigma_{\textbf{m}}\mid\textbf{m}\in(\mathbb{N}_s)^u, s\in\mathbb{N}\}$ can be organized as the nodes of a weighted spanning tree for $\mathcal{K}(\Theta_p)$.
Here, all the corresponding $\mathcal{K}(\Theta_p)$ are countable. If there is only one possible isomorphism type of
$l$-Berge-trees (for example if $d=l=1$ or $l=1$ and $d$ is arbitrary), then $\mathcal{K}(\Theta_p)$ has exactly one 
limit point, corresponding to having an infinite number of copies of Berge-trees of that type.

  \begin{figure}[!htb]
    \centering
    \includegraphics[height=20mm,width=40mm,angle=0]{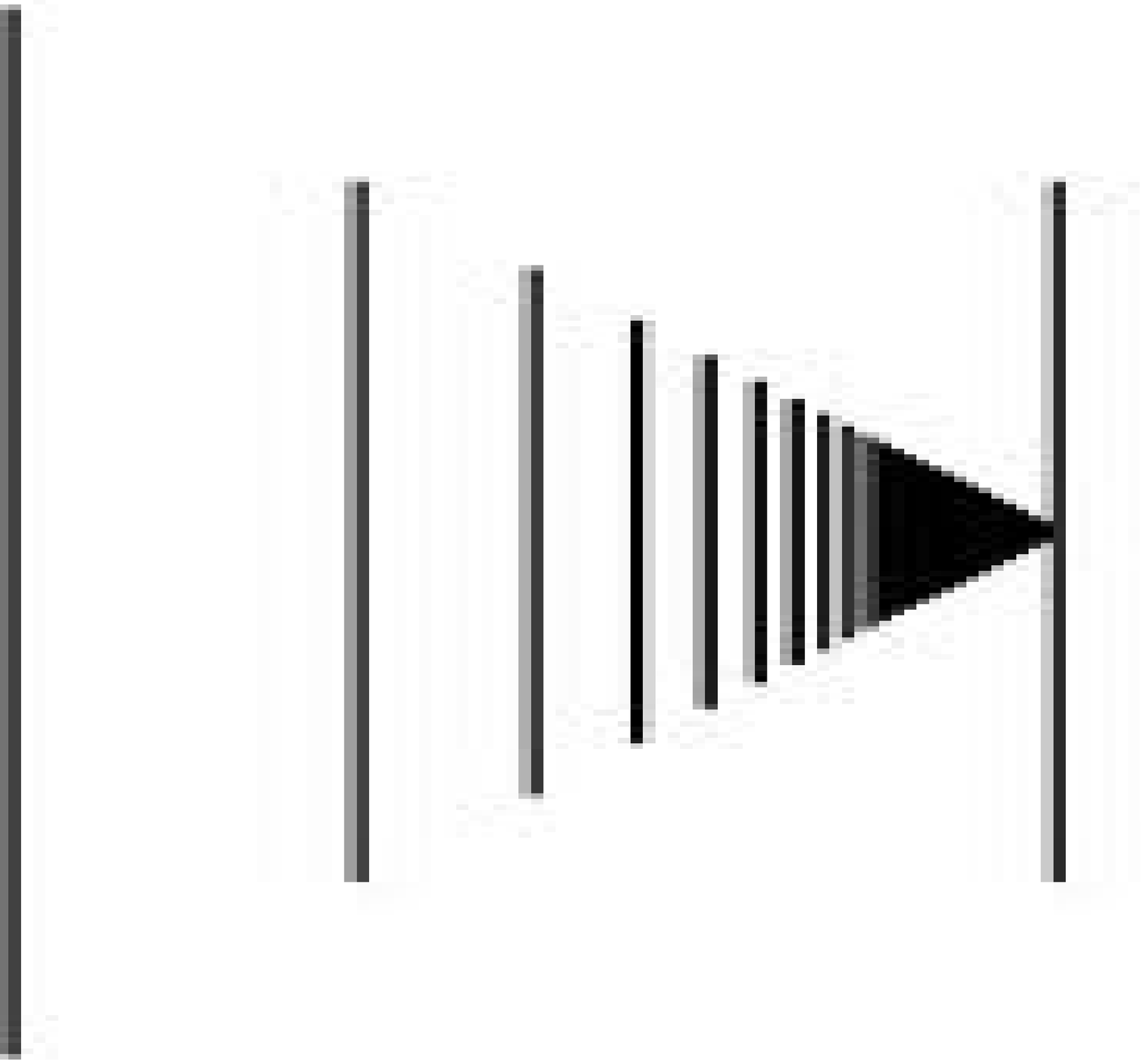}
    \caption{$\mathcal{K}(\Theta)=\omega+1$}
    \label{figRotulo}
  \end{figure}

If there are at least two isomorphism types, then there is a countably infinite number of limit points.
Each one
 corresponds to specifying finite quantities for the various isomorphism types of $l$-Berge-trees, except for one, whose copies are insisted 
to appear an infinite number of times.

  \begin{figure}[!htb]
    \centering
    \includegraphics[height=20mm,width=80mm,angle=0]{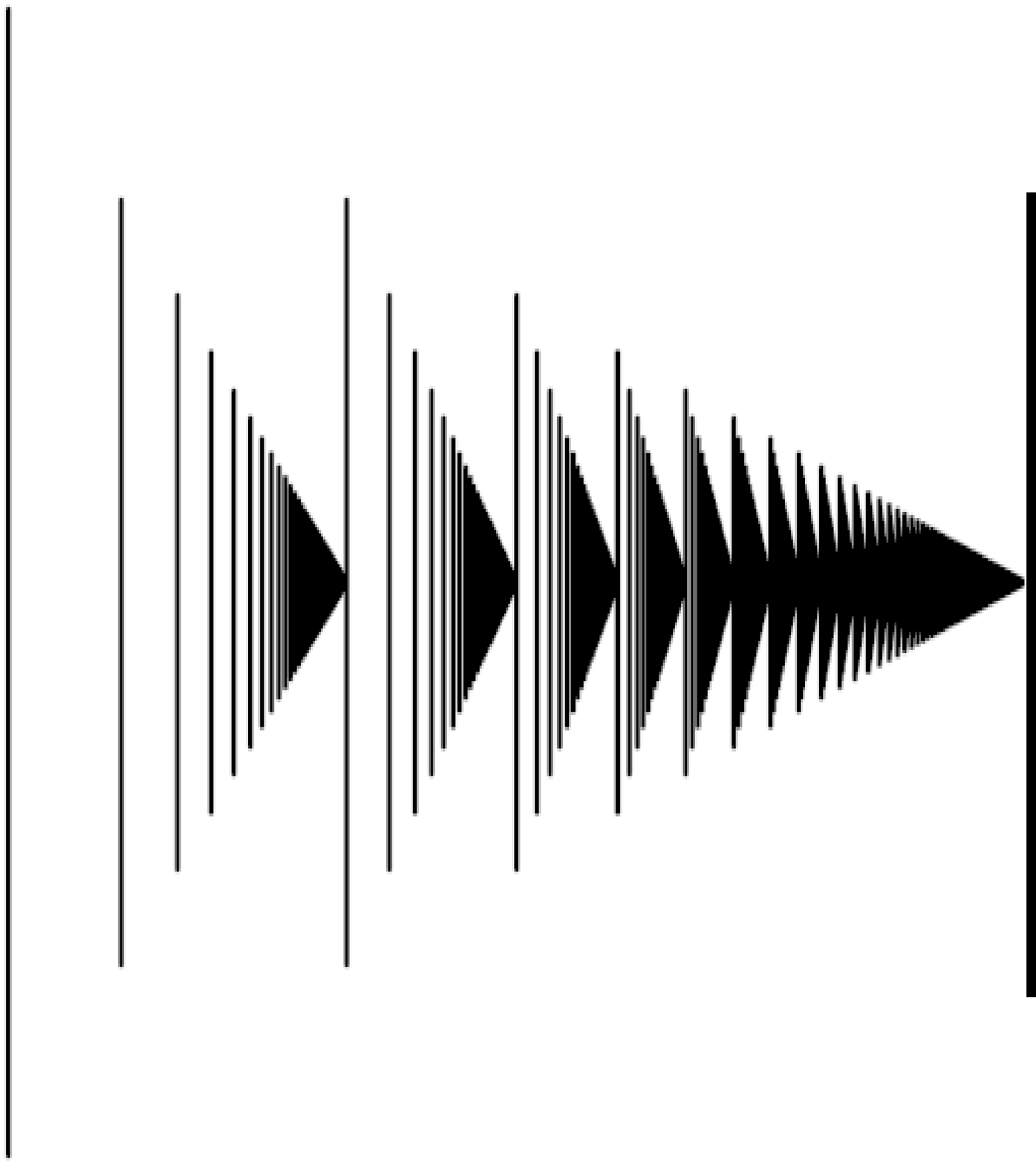}
    \caption{$\mathcal{K}(\Theta)=\omega^2+1$}
    \label{figRotulo}
  \end{figure}

It is worth noting that if $a_i$ is the number of automorphisms of the Berge-tree whose isomorphism type is $\tau_i$ then one has $\frac{c_i}{v!}=\frac{1}{a_i}$.

The convergence laws we got so far provide a nice description of the component structure in the early history of $G^{d+1}(n,p)$: it begins empty, then isolated edges appear, then all Berge-trees of order two, then all of order three, and so on untill $\ll n^{-d}$, immediately before the double jump takes place.

In what follows, $\bb$ stands for ``Big-Bang".

\begin{Def}

 $\bb$ is the set of all $L$-functions $p:\mathbb{N}\to[0,1]$ satisfying $0\le p\ll n^{-d} $.

\end{Def}

Now it is just a matter of putting pieces together to get the following theorem.

\begin{Thrm}

All elements of $\bb$ are convergence laws.

\end{Thrm}

\begin{proof}

Just note that any $L$-function on the above range satisfies one of the following conditions:

\begin{enumerate}

   \item $0\le p\le n^{-(d+1)}$
   \item $n^{-\frac{1+ld}{l}}\ll p\ll n^{-\frac{1+(l+1)d}{l+1}}$, for some $l\in\mathbb{N}$
   \item $n^{-(d+\epsilon)}\ll p\ll n^{-d}$ for all $\epsilon>0$
   \item $p\sim c\cdot n^{-\frac{1+ld}{l}}$ for some constant $c\in(0,+\infty)$

\end{enumerate}
\end{proof}


It is worth noting that the arguments used in getting zero-one laws for the intervals
\begin{enumerate}

   \item $0\le p\le n^{-(d+1)}$
   \item $n^{-\frac{1+ld}{l}}\ll p\ll n^{-\frac{1+(l+1)d}{l+1}}$, for some $l\in\mathbb{N}$
   \item $n^{-(d+\epsilon)}\ll p\ll n^{-d}$ for all $\epsilon>0$
   
 \end{enumerate}
 do not require the edge functions to be in Hardy's class, so \emph{all} functions inside those intervals are zero-one laws, regardless of being $L$-functions.
 
 On the other hand, taking $p=c(n)\cdot n^{-\frac{1+ld}{l}}$, where $c(n)$ oscillates infinitely
 often between two
 different  positive values is sufficient to rule out a convergence law for that edge function.
    
 Also we note that, in a certain sense, most of the functions in $\bb$ are zero-one laws: the only way one of that functions can avoid this condition is being inside
one of the countable windows inside a threshold of appearance of Berge-trees of some order. 

In the following sections, similar pieces of reasoning will yield an analogous result for another interval of edge functions.

     \section{Big-Crunch}

The present section is devoted to getting a result analogous to the ones above on another interval of edge functions, immediately after the double jump. We call that interval $\bc$, for Big-Crunch, because, informally, when ``time" (the edge functions $p$) flows forth, the behavior of the complement of the giant component is the same of the behavior $G^{d+1}(n,p)$ assumes in the Big-Bang $\bb$ with time flowing backwards.

More concretely, $\bc$ is the set of $L$-functions $p$ satisfying $n^{-d}\ll p\ll n^{-d+\epsilon}$ for all $\epsilon>0$. An important function inside this interval is $p=(\log n)n^{-d}$ which is known \cite{tese} to be 
the threshold for $G^{d+1}(n,p)$ to be connected. In the subintervals $n^{-d}\ll p\ll (\log n)n^{-d}$ and $(\log n)n^{-d}\ll p\ll n^{-d+\epsilon}$, nothing interesting happens in the first order perspective. This will imply that these intervals are entirely made of zero-one laws.

Inside the window $p\sim C\cdot(\log n)n^{-d}$ (with $C$ some positive constant), very much the opposite is true: here we find an infinite collection of local thresholds of elementary properties and also an infinite collection of zero-one and convergence laws.

\subsection{Just Past the Double Jump}

Consider the countable models of the almost sure theory $T_p$ with 

     $$n^{-d}\ll p\ll(\log n)n^{-d}.$$

 As we have already seen, in that range we still have components isomorphic to all
 finite Berge-trees of all orders and the possibility of infinite Berge-trees is still open.    
 The threshold for the appearance of small sub-hypergraphs excludes the possibility of bicyclic  
 (or more) components. By the same reason, we have components with cycles of all types.
 The following proposition shows, in particular, that vertices of small degree do not occur near the cycles.
 
 \begin{Prop}\label{unicyclic}

Suppose $p\gg n^{-d}$. Let $H$ be a finite connected configuration with at least one cycle and at least one vertex of small degree. Then the expected number of such configurations in $G^{d+1}(n,p)$ is $o(1)$.
In particular a.a.s. there are no such configurations.

\end{Prop}

\begin{proof}

Let the configuration $H$ have $v$ vertices and $l$ edges. As $H$ is connected and has at least one cycle, we have $v\leq ld$. For convenience, set $\alpha=\frac{pn^d}{d!}$. Note that $\alpha\to +\infty$.
Let $E$ be the expected number of configurations $H$. Then

\begin{align*}
E&= O\left(\frac{n^v}{v!}p^l(1-p)^{\frac{n^d}{d!}}\right)= O\left(\frac{n^v}{v!}p^l\exp(-p\frac{n^d}{d!})\right) \\
&=O\left(n^{dl}p^l\exp(-p\frac{n^d}{d!})\right)t=O\left(\alpha^l\exp(-\alpha)\right)=o(1).
\end{align*}
The last part follows from the first moment method.
\end{proof}

Now it is easy to see that the edge functions in the present range are zero-one laws, getting part $(1)(d)$ of \ref{grandeteorema}.

\begin{Prop}

Suppose $p$ is an edge function satisfying
                $$n^{-d}\ll p\ll(\log n)n^{-d}.$$
        Then $p$ is a zero-one law.
               
\end{Prop}

\begin{proof}

By Proposition \ref{unicyclic}, every vertex in the union of all the unicyclic components has infinite neighbors. This determines these components up to isomorphism and it does not pay for Spoiler to play there.
But in the complement of the above set, we have already seen that Duplicator can win all
$k$-round Ehrenfeucht Games. Therefore all countable models of $\Theta_p$ are elementarily
equivalent and these $p$ are zero-one laws.
\end{proof}

We note that the non-existence of vertices of small 
degree near cycles is first-order axiomatizable. For each $l,s,k\in\mathbb{N}$ there is a first order sentence which excludes all of the (finitely many)
 configurations with cycles of order $\le l$ at distance $\le s$ from one vertex of degree $\le k$. Similar considerations show that the non-existence of bicyclic (or more) components is also first-order
axiomatizable. So one easily gets a simple axiomatization for the almost sure theories of the
above edge functions.

\subsection{Beyond Connectivity}

Now we consider countable models of $\Theta_p$ with 

                     $$(\log n)n^{-d}\ll p\ll n^{-d+\epsilon}$$
  for all positive $\epsilon$.
  
  Again, the thresholds for appearance of small sub-hypergraphs imply that, in this range, we 
  have all cycles of all types as sub-hypergraphs, and no bicyclic (or more) components.
  Now all vertices of small degree are gone.
  
  \begin{Prop}
  
  For $p\gg (\log n)n^{-d}$, the expected number of vertices of small degree in $G^{d+1}(n,p)$
  is $o(1)$. In particular, a.a.s. there are no vertices of small degree.
  
   \end{Prop}

   \begin{proof}
   
   Fix a natural number $k$ and let $E$ be the expected number of vertices of degree
   $k$ in $G^{d+1}(n,p)$.  Then
   
        $$E\sim n(1-p)^{\frac{n^d}{d!}}\sim n\exp\left({-p\frac{n^d}{d!}}\right)=o(1).$$
   \end{proof} 
   
   The following proposition gives part $(i)(h)$ of \ref{grandeteorema}.
   
\begin{Prop} \label{beyond conectivity}

Let $p$ be an edge function satisfying
              $$(\log n)n^{-d}\ll p\ll n^{-d+\epsilon}$$
         Then $p$ is a zero-one law.

\end{Prop}
   
\begin{proof}

   The countable models of $\Theta_p$ have components that contain cycles of all types, no bicyclic 
   (or more) components and may possibly have Berge-tree components. As no vertex can have
   small degree, all vertices in that components have infinite neighbors, so these components
   are unique up to isomorphism. But $\Theta_p$ is not $\aleph_0$-categorical since the existence of Berge-tree components is left open.
 The results concerning wining strategies for Duplicator mentioned before  give that these models are elementarily equivalent, so these $p$
 are zero-one laws. 
  \end{proof}
 
 The discussion found on the proof of Theorem \ref{beyond conectivity}  also gives simple axiomatizations for the almost sure theories
 of the above $p$.

         \subsection{Marked Berge-Trees}  
         
         Now we are left to the case of $L$-functions $p$ comparable to $n^{-d}\log n$. In other words, to complete our discussion, we must get a description of what happens when an edge function $p$ is such that $n^d p\slash \log n$ tends to a finite constant $C\neq 0$.
         
         In the last section, the counting of the connected components isomorphic to Berge-trees was the 
   fundamental piece of information in the arguments that implied all the convergence laws we found there. Copies of 
   Berge-trees as connected components are, in particular, induced such copies.
         
       It turns out that the combinatorial structure whose counting is fundamental to getting
  the convergence laws in the window $p\sim C\cdot\frac{\log n}{n^d}$ is still that of Berge-trees, but
  now the copies are not necessarily induced. Instead, some vertices receive
  markings, meaning that those vertices must have no further neighbors than those indicated on
  the ``model" Berge-tree. On the non-marked vertices no such requirement is imposed: they are
  free to bear further neighbors. These copies of Berge-trees are then, in a sense, ``partially induced".

\begin{Def}

Let $v^*,l\in\mathbb{N}$. A $v^*$-marked $l$-\emph{Berge-tree} is a finite connected (Berge)-acyclic hypergraph with $l$ edges and with $v^*$ distinguished vertices, called the \emph{marked} vertices.

\end{Def}

  \begin{figure}[!htb]
    \centering
    \includegraphics[height=60mm,angle=90]{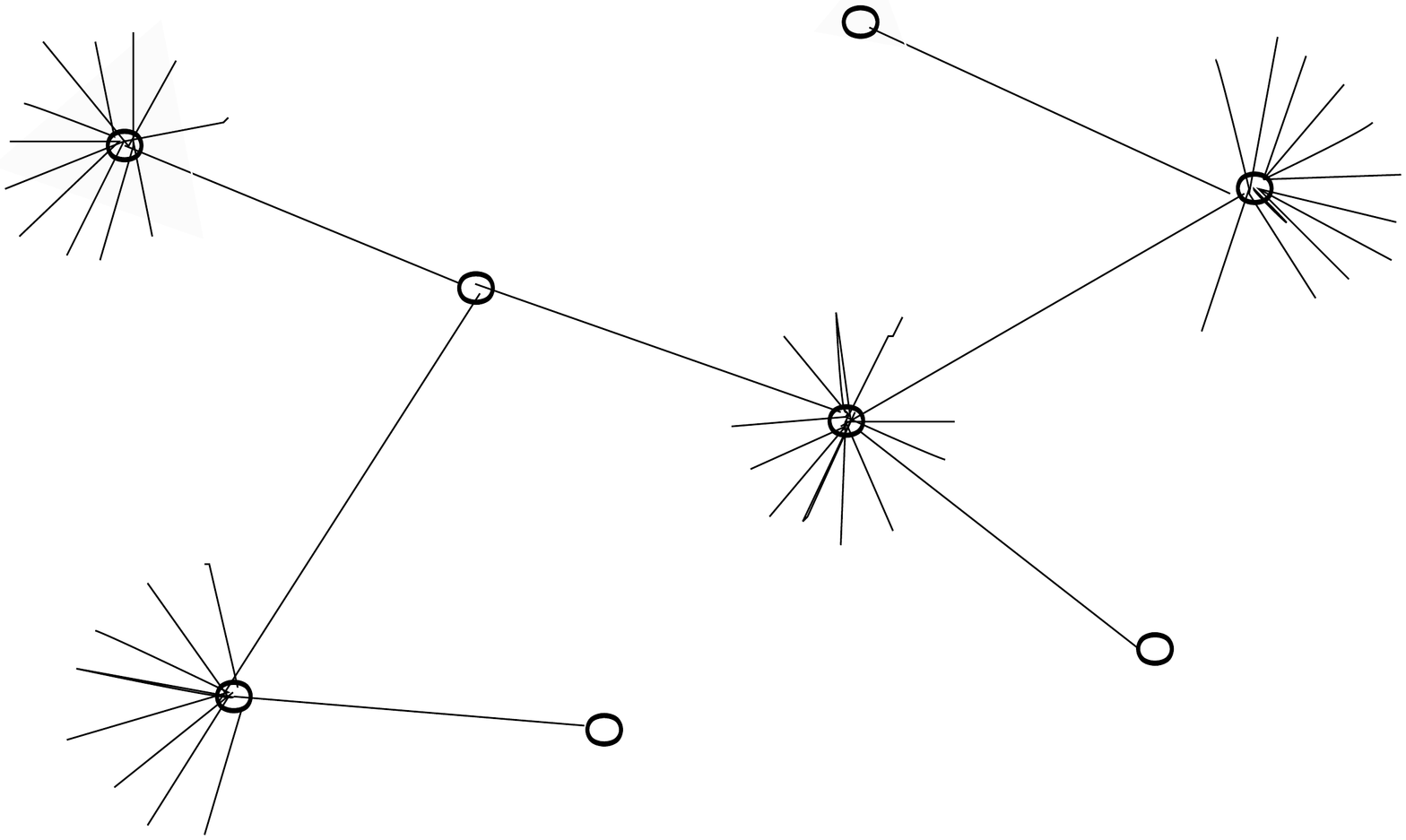}
    \caption{A marked Berge-tree}
   \label{figRotulo}
  \end{figure}


Note that a $v^*$-marked $l$-Berge-tree is a hypergraph on $v=1+ld$ vertices.

\begin{Def}

Let $B$ be a $v^*$-marked $l$-Berge-tree and $H$ be a hypergraph. A \emph{copy} of $B$ in $H$ is a (not necessarily induced) sub-hypergraph of $H$ isomorphic to $B$ where if $w$ is a marked vertex of $B$ and $w'$ is the corresponding vertex of $H$ under the above isomorphism, then $w$ and $w'$ have the same degree.

\end{Def}

An edge of a Berge-acyclic hypergraph incident to exactly one other edge is called a \emph{leaf}.

\begin{Def}

A $v^*$-marked $l$-Berge-tree is called \emph{minimal} if every leaf has at least one marked vertex.

\end{Def}

  \begin{figure}[!htb]
    \centering
   \includegraphics[height=60mm,angle=90]{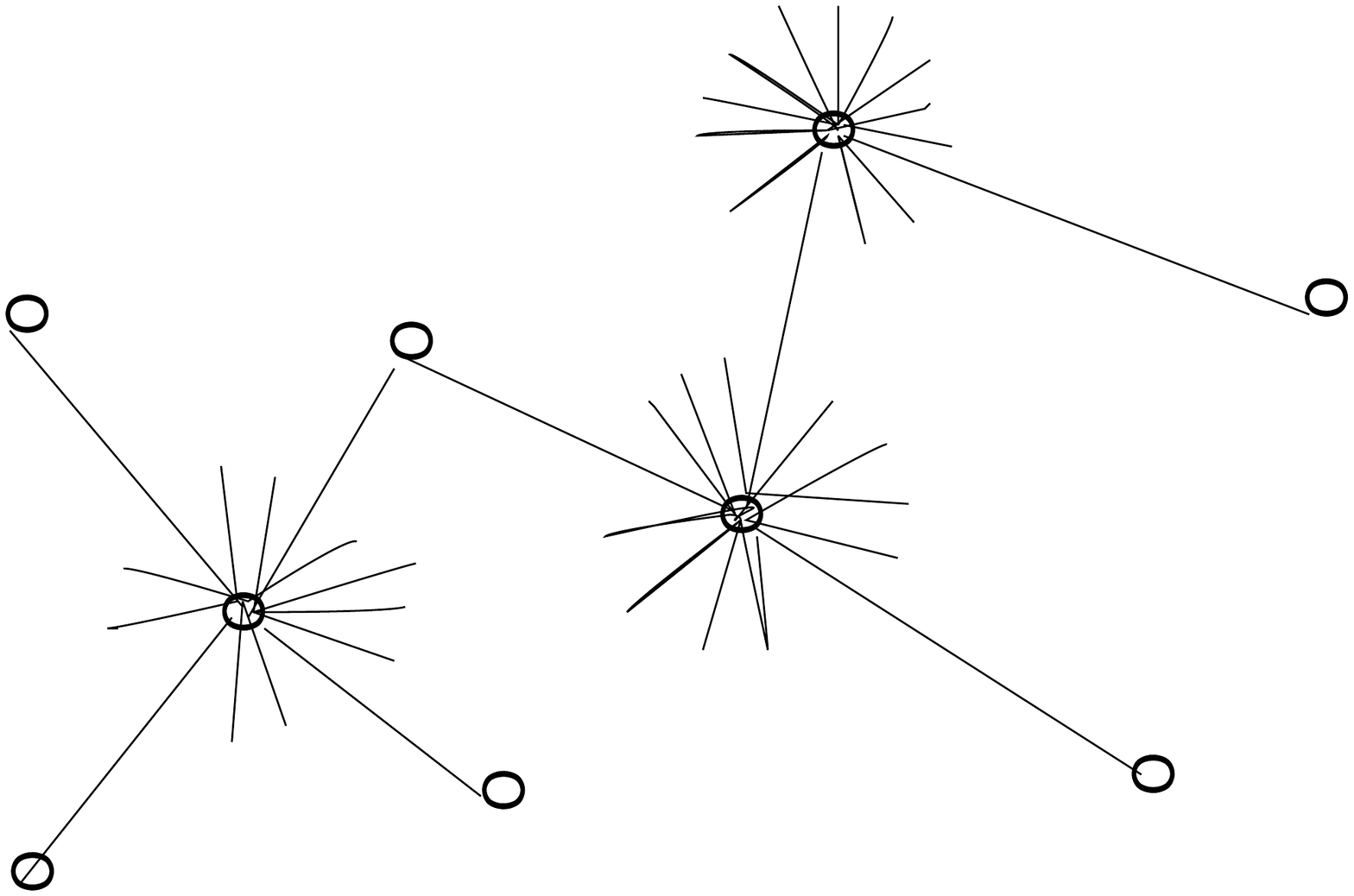}
    \caption{A minimal marked Berge-tree}
    \label{figRotulo2}
  \end{figure}   
  
  \begin{figure}[!htb]
    \centering
   \includegraphics[height=60mm,angle=90]{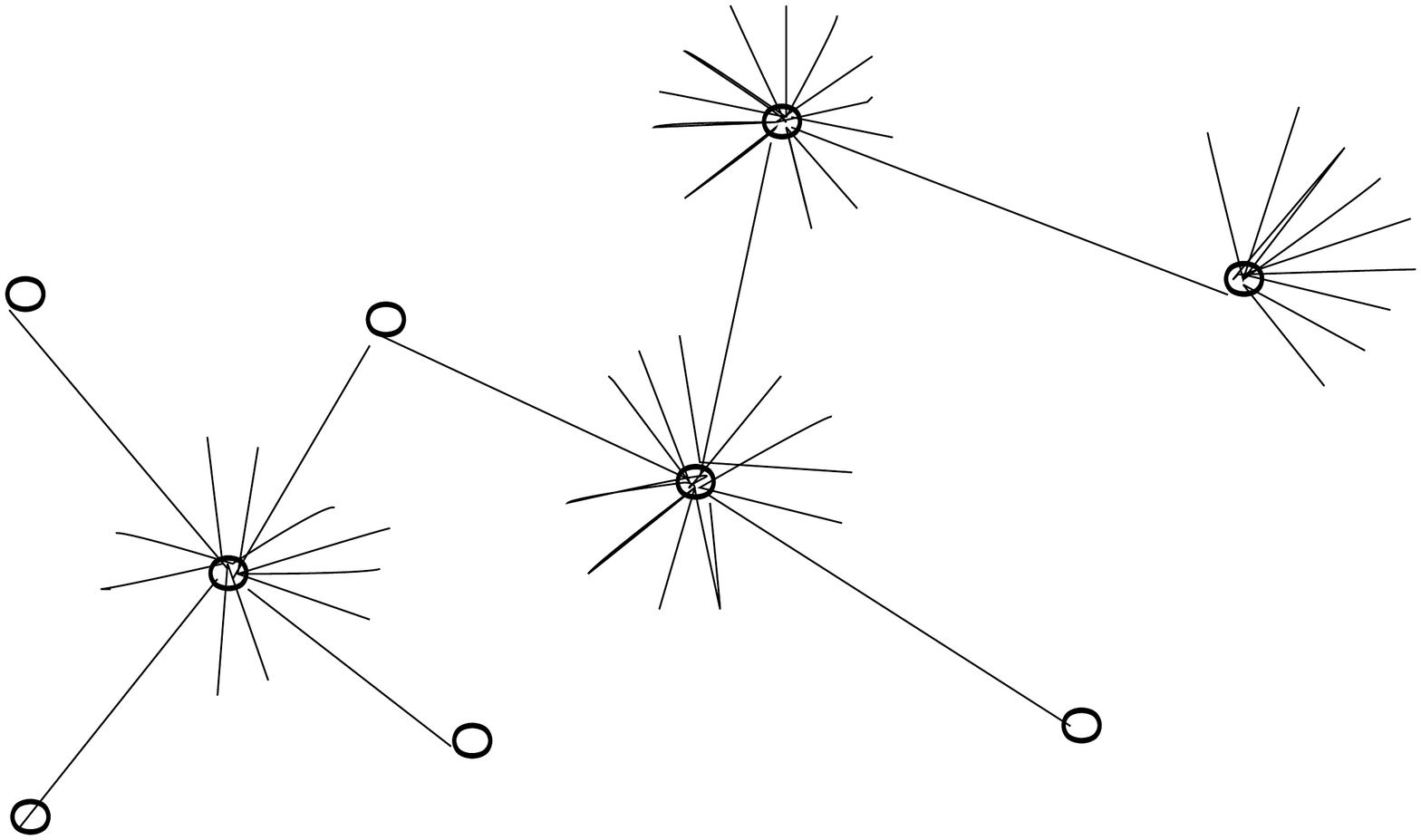}
    \caption{A non-minimal marked Berge-tree}
    \label{figRotulo2}
  \end{figure}

Now, the most important concept to understanding the zero-one and convergence laws on $\bc$ is the counting of minimal marked Berge-trees.

\begin{Def}

Let $\Gamma$ be the finite set of all isomorphism types of minimal $v^*$-marked $l$-Berge-trees on $1+ld$ labelled vertices and fix $\gamma\in\Gamma$.
Then  $c(l,v^*,\gamma)$ is the number of possible $v^*$-marked $l$-Berge-trees of isomorphism type $\gamma$ on $1+ld$ labelled vertices.

\end{Def}

\begin{Def}

The random variable $A(l,v^*,\gamma)$ is the number of copies of $v^*$-marked $l$-Berge-trees of isomorphism type $\gamma$ in $G^{d+1}(n,p)$.

\end{Def}

                      \subsubsection{Counting of Marked Berge-Trees}
                      
 Now we use the first and second moment methods to get precise information on the 
 counting of minimal marked Berge-trees for edge functions on the range
  
                           $$p\sim C\cdot\frac{\log n}{n^d} \ , \  C>0.$$
                           
Rather informally, when the coefficient of $\frac{\log n}{n^d}$ in $p$ avoids the rational value
$\frac{d!}{v^*}$ then the expected number of $v^*$-marked Berge-trees is either $0$ or $\infty$. The first moment method implies that, in the first case, a.a.s. there are no $v^*$-marked Berge-trees.
The second moment method will yield that, in the second case, there are many such minimal marked Berge-trees.

If $C=\frac{d!}{v^*}$, then knowledge of more subtle behavior of the edge function is required:
we are led to consider the coefficient $\omega$ of $\frac{\log\log n}{n^d}$ in $p$. If this coefficient
avoids the integer value $l$ then the expected number of $v^*$-marked $l$-Berge-trees is either
$0$ or $\infty$. Again, first and second moment arguments imply that, in the first case, the number 
of such Berge-trees is a.a.s. zero and, in the second case, the number of such minimal Berge-trees is very large.

Finally, if $\omega=l$, then knowledge of even more subtle behavior of the edge function is 
required: we consider the coefficient $c$ of $\frac{1}{n^d}$ in $p$. If this coefficient diverges,
then the expected number of $v^*$-marked $l$-Berge-trees is either $0$ or $\infty$ and, again,
first and second moment methods imply that the actual number of such Berge-trees is what 
one expects it to be.

All above cases give rise to zero-one laws. The remaining case is the one when the coefficient
$c$ converges. In this case, the fact that the almost sure theories are almost complete
will yield convergence laws.


Let $p=p(n)$ be comparable to $\frac{\log n}{n^d}$. That is, let $\frac{n^d p}{\log n}$ converge to a constant $C\neq 0$.

\begin{Prop}

Fix $\gamma\in\Gamma$.

  \begin{enumerate}

             \item If $C<\frac{d!}{v^*}$ then $\mathbb{E}_n[A(l,v^*,\gamma)]\to +\infty$.

             \item If $C>\frac{d!}{v^*}$ then $\mathbb{E}_n[A(l,v^*,\gamma)]\to 0$ for all $l\in\mathbb{N}$.

  \end{enumerate}

In particular, if $C>\frac{d!}{v^*}$ then, for any $l\in\mathbb{N}$, a.a.s. $A(l,v^*,\gamma)=0$.

\end{Prop}

\begin{proof}

Set $c=c(l,v^*,\gamma)$ and $v=1+ld$.

Note that $pv^*\frac{n^d}{d!}\sim Cv^*\frac{\log n}{d!}$ so $pv^*\frac{n^d}{d!}- Cv^*\frac{\log n}{d!}=o(1)\log n$. Therefore one has

\begin{align*}
&\mathbb{E}_n[A(l,v^*,\gamma)]\sim 
c\frac{n^v}{v!}p^l(1-p)^{v^*\frac{n^d}{d!}} 
\sim c\frac{n^v}{v!}p^l\exp\left(-p v^*\frac{n^d}{d!}\right) \\
&\sim c\frac{n^v}{v!}p^l\exp\left(o(1)\log n-Cv^*\frac{\log n}{d!}\right) \\
&\sim c\frac{n^v}{v!}(C\log n)^l n^{-ld}\exp\left(o(1)\log n-Cv^*\frac{\log n}{d!}\right) \\
&\sim\frac{c}{v!}(C\log n)^l n^{1-\frac{Cv^*}{d!} + o(1)},
\end{align*}

and the result follows.

The last part follows from the first moment method.
\end{proof}

 Now consider $p\sim\frac{d!}{v^*}\cdot\frac{\log n}{n^d}$ so that $v^*n^d\frac{p}{d!}-\log n=o(1)\log n$ and let 

$$\omega(n)=\frac{v^*n^d\frac{p}{d!}-\log n}{\log\log n}.$$ 

\begin{Prop}

 Fix $l\in\mathbb{N}$ and $\epsilon>0$.
 
                       \begin{enumerate}
           
          \item If eventually $\omega<l-\epsilon$ then $\mathbb{E}_n[A(l,v^*,\gamma)]\to +\infty$

          \item If eventually $\omega>l+\epsilon$ then $\mathbb{E}_n[A(l,v^*,\gamma)]\to 0$.

                        \end{enumerate}
In particular, the second condition implies that a.a.s. $A(l,v^*,\gamma)=0$.

\end{Prop}

\begin{proof} 

Set $c:=c(l,v^*,\gamma)$ and $v:=1+ld$.

Note that  $$v^*n^d\frac{p}{d!}=\log n+\omega\log\log n$$ 
so one has

\begin{align*}
&\mathbb{E}_n[A(l,v^*,\gamma)]\sim \frac{c}{v!}n^v p^l(1-p)^{v^*\frac{n^d}{d!}} 
\sim \frac{c}{v!}n^v p^l\exp\left(-p v^*\frac{n^d}{d!}\right) \\
&\sim\frac{c}{v!}n^v p^l\exp(-\log n-\omega\log\log n) \\
&\sim \frac{c}{v!}n^v\left(\frac{d!}{v^*}\log n\right)^ln^{-ld}n^{-1}(\log n)^{-\omega} 
\sim\frac{c}{v!}\left(\frac{d!}{v^*}\right)^l(\log n)^{l-\omega}
\end{align*}

 and the result follows. 
 
The last part follows from the first moment method.
\end{proof}

Now consider the case $\omega\to l\in\mathbb{R}$ and let

            $$c(n):=p\frac{n^dv^*}{d!}-\log n-l\log\log n.$$

 \begin{Prop}
 
 Fix $\gamma\in\Gamma$ and $c=c(n)$ as above.
     
    \begin{enumerate}
    
         \item If $c\to-\infty$ then $\mathbb{E}_n[A(l,v^*,\gamma)]\to+\infty$
         \item If $c\to+\infty$ then $\mathbb{E}_n[A(l,v^*,\gamma)]\to 0$.

    \end{enumerate}
In particular, the second condition implies that a.a.s. $A(l,v^*,\gamma)=0$.

 \end{Prop}
 
 \begin{proof}
 
 Note that $p\frac{n^dv^*}{d!}=\log n+l\log\log n+c(n)$, so 
 
 \begin{align*}
&\mathbb{E}_n[A(l,v^*,\gamma)]\sim \frac{c(l,v^*,\gamma)}{v!}n^v p^l(1-p)^{v^*\frac{n^d}{d!}} 
\sim \frac{c(l,v^*,\gamma)}{v!}n^v p^l\exp\left(-p v^*\frac{n^d}{d!}\right) \\
&\sim\frac{c(l,v^*,\gamma)}{v!}n^v p^l\exp(-\log n-l\log\log n-c(n)) \\
&\sim\frac{c(l,v^*,\gamma)}{v!}\left(\frac{d!}{v^*}\right)^l\exp(-c(n)),
\end{align*}
and $1$ and $2$ follow.  
     
     The last part follows from the first moment method.
  \end{proof}


Fix, in $G^{d+1}(n,p)$, a vertex set $S$ of size $|S|=1+ld$ and $\gamma\in\Gamma$. To each of the $c:=c(l,v^*,\gamma)$ potential copies of $v^*$-marked $l$-Berge-trees of type $\gamma$ in $S$ there corresponds
the random variable $X_\alpha$, the indicator of the event $B_\alpha$ that this potential copy is indeed there in $G^{d+1}(n,p)$. Then we clearly have

               $$A(l,v^*,\gamma)=\sum_\alpha X_\alpha.$$

We write $|X_\alpha|:=S$.

\begin{Prop} \label{intersectiontype}

Let $p\sim C\cdot\frac{\log n}{n^d}$, where $\frac{d!}{v^*+1}< C<\frac{d!}{v^*}$. Then, for any $k,l\in\mathbb{N}$, we have $$\mathbb{P}_n[A(l,v^*,\gamma)\geq k]\to 1.$$

\end{Prop}

\begin{proof}

By the first moment analysis, the condition on the hypothesis implies $\mathbb{E}_n[A(l,v^*,\gamma)]\to +\infty$ and $\mathbb{E}_n[A(\tilde{l},\tilde{v}^*,\gamma)]\to 0$ for all $\tilde{v}^*>v^*$ and any $\tilde{l}\in\mathbb{N}$. We use the second moment method. As

$$\mathbb{E}_n\left[\sum_{|X_\alpha|\cap |X_\beta|=\emptyset} X_\alpha X_\beta\right]\sim c^2\frac{n^{2v}}{v!^2}p^{2l}(1-p)^{2v^*\frac{n^d}{d!}}\sim\mathbb{E}_n[A(l,v^*,\gamma)]^2$$

it suffices to show that  

$$\mathbb{E}_n\left[\sum_{|X_\alpha|\cap |X_\beta|\ne\emptyset} X_\alpha X_\beta\right]=o(1).$$  

The sets $|X_\alpha|$ and $|X_\beta|$ can only  intersect according to a finite number of patterns,
so it suffices to show that the contribution of all terms with a given pattern is $o(1)$.
Set $S:=|X_\alpha|\cup |X_\beta|$.

Consider an intersection type $S$ such that the model spanned by $S$ contains a cycle. Then the 
configuration $S$ has a vertex of small degree (marked) near a cycle. The sum of contributions of all terms with that intersection type is $\sim$ the expected number of such configurations. As there
is a vertex of small degree near a cycle, this is $o(1)$ by proposition $35$.

If the type of $S$ has no cycles, then $S$ is a maked Berge-tree with $\tilde{v}^*\geq v^*$ marked vertices.

   If $\tilde{v}^*>v^*$ then, by the first moment analysis, the sum of contributions of those terms is 
   $o(1)$.

    We claim that there are no terms with $\tilde{v}^*=v^*$. Indeed, if that were the case,  by minimality, all edges would be in the intersection and so the events indicated by $X_\alpha$ and
    $X_\beta$ would be the same, a contradiction.
\end{proof}

Now consider $p\sim\frac{d!}{v^*}\cdot\frac{\log n}{n^d}$ and, as above, let 

$$\omega(n)=\frac{v^*n^d\frac{p}{d!}-\log n}{\log\log n}.$$

\begin{Prop}

Fix $\epsilon>0$ and $l\in\mathbb{N}$.

    \begin{enumerate}

      \item If eventually $\omega<l-\epsilon$ then for any $k\in\mathbb{N}$, we have $$\mathbb{P}_n[A(l,v^*,\gamma)\geq k]\to 1$$
      \item If $\omega\to+\infty$ then for any $k,\tilde{l}\in\mathbb{N}$, we have $$\mathbb{P}_n[A(\tilde{l},v^*-1,\tilde{\gamma})\geq k]\to 1$$ for all isomorphism types 
      $\tilde{\gamma}$ of minimal $(v^*-1)$-marked \\ $\tilde{l}$-Berge-trees.
      
      \end{enumerate}

\end{Prop}

\begin{proof}

 The proof of $1$ is the same as the proof of the above proposition.
 
 The proof of $2$ is analogous, noting that condition $2$ implies that the expected number of
 $v^*$-marked Berge-trees with any fixed number of edges is $o(1)$, so that the intersection pattern must have all the $v^*-1$ marked vertices.
\end{proof}

Now consider the case $\omega\to l$ and, as above, let

     $$c(n)=\frac{pn^dv^*}{d!}-\log n-l\log\log n.$$
     
     Finally, the same reasoning used in the proofs of the two above propositions demonstrates the following proposition.

\begin{Prop}

If $c\to-\infty$ then $\mathbb{P}_n[A(l,v^*,\gamma)\geq k]\to 1$ for any $k\in\mathbb{N}$.

\end{Prop}

              \subsection{Zero-One Laws Between the Thresholds}

       Now we consider the countable models of the almost sure theories $\Theta_p$
  for $p$ ``between" the critical values above, and get parts $(i)(e),(f),(g)$ of \ref{grandeteorema}.
  
  \begin{Prop}
  
  Let $p$ be an edge function satisfying one of the following properties:
  
   \begin{enumerate}

                         \item $p\sim C\cdot\frac{\log n}{n^d}$, where $\frac{d!}{v^*+1}<C<\frac{d!}{v^*}$ for  some $d,v^*\in\mathbb{N}$.  
                         \item $p\sim \frac{d!}{v^*}\cdot\frac{\log n}{n^d}$ where $\omega\to\pm\infty$ or
            $\omega\to C$, where $l-1<C<l$ for some $l\in\mathbb{N}$
                         \item $p\sim\frac{d!}{v^*}\cdot\frac{\log n+l\log\log n}{n^d}$, where $c\to\pm\infty$

    \end{enumerate}    
    Then $p$ is a zero-one law.                 
  
  \end{Prop}
  
  \begin{proof}
  
        Consider, first, a function $p\sim C\cdot\frac{\log n}{n^d}$, where
                       
                       $$\frac{d!}{v^*+1}<C<\frac{d!}{v^*}.$$
        We consider the models of the almost sure theory $\Theta_{v^*}:=\Theta_p$.             
       In that range, we still have no bicyclic (or more) components in the first-order perspective. As there are no vertices of 
       small degree near cycles, the unicyclic components are determined up to isomorphism.
       Also we still have infinitely many copies of each cycle. So the union of connected components 
       containing cycles are determined up to isomorphism and Duplicator does not have to
       worry about them: every time Spoiler plays there, he has wasted a move.
       
So let us consider the Berge-tree components. By the first and second moment analysis above,
we have no components containing $(v^*+1)$-marked Berge-trees and have infinite components containing copies  of each minimal $v^*$-marked Berge-tree. Each component containing a $v^*$-marked Berge-tree is determined
up to isomorphism: each non-marked vertex must have infinite neighbors.

Let $l\in\mathbb{N}$ such that $1+ld\le v^*<v^*+1\le 1+(l+1)d$. 
Then there are no Berge-trees of order $l+1$ (or more) as sub-hypergraphs and there are infinitely many
components isomorphic to each Berge-tree of order $\le l$.
Therefore, the union of the components isomorphic to finite Berge-trees is determined up to isomorphism.

The theory $\Theta_{v^*}$ is not $\aleph_0$-categorical though, since in that countable models, there may or may
not be components containing $\tilde{v}^*$-marked Berge-trees with $\tilde{v}^*<v^*$. 
(This includes the degenerate case $\tilde{v}^*=0$: there may or may not be infinite Berge-trees 
where all vertices have infinite neighbors)
These 
components are ``simulated" by components containing $v^*$-marked vertices, with
$v^*-\tilde{v}^*$ marked vertices suitably far from the neighborhood of the $\tilde{v}^*$ marked vertices, this neighborhood being a copy of the $\tilde{v}^*$-marked Berge-tree one wants to
simulate.

More precisely, the countable models of $\Theta_p$ satisfy the hypothesis of 
Proposition \ref{winning strategy}, so they are pairwise elementarily equivalent and, hence, $\Theta_{v^*}$ is complete and the corresponding $p$ are 
zero-one laws.

Now consider $p\sim\frac{d!}{v^*}\cdot\frac{\log n}{n^d}$ and, as above, let 

$$\omega(n)=\frac{v^*n^d\frac{p}{d!}-\log n}{\log\log n}.$$ 

If $\omega\to-\infty$ then the first and second moment analysis above imply that the countable models of $\Theta_p$ 
are the same as the countable models of $\Theta_{v^*}$ and, as $\Theta_{v^*}$ is complete, $p$ is a zero-one law.

If $\omega\to+\infty$ then the first and second moment analysis above imply that the countable models of 
$\Theta_p$ are the same as the countable models of $\Theta_{v^*-1}$. But the latter theory is complete and, hence, the corresponding $p$ are zero-one laws.

 If $\omega\to C$, with $l-1<C<l$, then the countable models of $\Theta^l_{v^*}:=\Theta_p$ are the same as the countable models of $\Theta_{v^*}$ but without the components with marked Berge-trees of 
order $\le l-1$. These models are, for the same reasons, still pairwise elementarily equivalent,
so we have that the corresponding $p$ are zero-one laws.

Finally, consider the case $\omega\to l$ and, as above, let

                 $$c(n)=\frac{pn^dv^*}{d!}-\log n-l\log\log n.$$

     If $c\to-\infty$, then the analysis above show that the countable models of $\Theta_p$ are the same
  as the countable models of $\Theta^l_{v^*}$, so these $p$ are zero-one laws.
  
    If $c\to+\infty$, then the analysis above show that the countable models of $\Theta_p$ are the same
  as the countable models of $\Theta^{l-1}_{v^*}$, so these $p$ are also zero-one laws.
  \end{proof}

             \subsection{Axiomatizations}

    At this point, it is clear that the arguments given in the last section actually give axiomatizations
    for the almost sure theories presented there.
    
    More formally, let the theory $\Theta(v^*)$ consist of a scheme of  axioms saying that there are no bicyclic (or more) components, a
    scheme of axioms saying that there are no copies of $\tilde{v}^*$-marked Berge-trees for each
    $\tilde{v}^*>v^*$ and a scheme of axioms saying that there are infinitely many copies of each minimal
    $v^*$-marked Berge-tree.  
    
    Similarly, let the theory $\Theta(v^*,l)$ consist of a scheme of axioms saying that there are no bicyclic (or more)
    components, a scheme of axioms excluding the $\tilde{v}^*$-marked Berge-trees for each 
    $\tilde{v}^*>v^*$, a scheme of axioms saying that there are no $v^*$-marked Berge-trees of 
    order $\le l-1$ and an scheme saying that there are infinitely many copies of each minimal $v^*$-marked Berge-tree not excluded by the last scheme.      
    
    By the discussion found in the last section, we have the following:
    
    \begin{Thrm}
    
     The theory $\Theta(v^*)$ is an axiomatization for $\Theta_{v^*}$ and, similarly, the theory $\Theta(v^*,l)$ is an axiomatization for
    $\Theta^l_{v^*}$.
    
    \end{Thrm}

             \subsection{On the thresholds}

 The only way an $L$-function can avoid all of the clauses discussed above is the possibility that
 $c(n)$ converges to a real number $c$. That is to say, we must consider the possibility that
 
         $$p=\frac{d!}{v^*}\cdot\frac{\log n+l\log\log n+c(n)}{n^d}$$
         where $c(n)\to c$.       
             
 We will see, in the present chapter, that these $p$, although not zero-one laws, are still convergence laws. The situation is analogous to that of the last section:
 on these thresholds, the spaces $K(\Theta_p)$ admit weighted spanning trees.

    \subsubsection{Limiting Probabilities on the Thresholds}

Let $v^*,l\in\mathbb{N}$ and let $T_1,T_2,\ldots,T_u$ denote the collection of all possible (up to isomorphism) $v^*$-marked Berge-trees of order $l$ and, for a $u$-tuples $\mathbf{m}=(m_1,\ldots,m_u)\in(\mathbb{N}_s)^s$, 
 let $\sigma_{\textbf{m}}$ be the elementary property that there are precisely $m_i$ components $T_i$ if $0\leq i\leq s$, and at least $s+1$ components $T_i$ if $m_i=\mathcal{M}$. 
Now we get part $(ii)(b)$ of \ref{grandeteorema}.

\begin{Prop}

Let  $p=\frac{d!}{v^*}\cdot\frac{\log n+l\log\log n+c(n)}{n^d}$, where $c(n)\to c$. Then the collection $\{\sigma_{\mathbf{m}}\mid\mathbf{m}\in(\mathbb{N}_s)^u, u\in\mathbb{N}\}$ is the set of nodes of a 
weighted spanning tree for $\mathcal{K}(\Theta_p)$. In particular, $p$ is a convergence law.

\end{Prop}

\begin{proof}

It is clearly enough to suppose that $\textbf{m}\in(\mathbb{N})^u$.

We claim the countable models of $\Theta_p\cup\{\sigma_{\textbf{m}}\}$ are pairwise elementarily
equivalent. Indeed, the complement of the union of components containing the $v^*$-marked
Berge-trees of order $l$ is elementarily equivalent to the countable models of the theory 
$\Theta_{v^*}^{l+1}$, defined above. As the latter theory is complete, $\Theta_p$ is also complete.

Tautologically no two of the $\sigma_{\textbf{m}}$ can hold simultaneously.

For each $i\in\{1,2,\ldots,u\}$, let $\delta_i$ be the isomorphism type of $T_i$. For notational convenience, set $c_i:=c(l,v^*,\delta_i)$ and $A_i:=A(l,v^*,\delta_i)$. 
The next lemma implies the remaining properties and, therefore, completes the proof.
\end{proof}

\begin{Lemma}  \label{poisson2}

In the conditions of the above proposition, the random variables $A_1,A_2,\ldots,A_u$ are asymptotically independent Poisson with means 

            $$\lambda_i=\frac{c_i}{v!}\left(\frac{d!}{v^*}\right)^le^{-c}.$$ That is to say,

       $$p_{\textbf{m}}:=\lim_{n\to\infty}\mathbb{P}_n[\sigma_{\textbf{m}}]=\prod^u_{i=1}e^{-\lambda_i}\frac{\lambda_i^{m_i}}{m_i!}.$$
In particular  $$\sum_{\textbf{m}\in I}p_{\textbf{m}}=1.$$

\end{Lemma}

\begin{proof}

By the method of factorial moments, is suffices to show that, for all $r_1,r_2,\ldots,r_u\in\mathbb{N}$ we have

                    $$\mathbb{E}_n\left[(A_1)_{r_1}\cdots(A_u)_{r_u}\right]\to\lambda^{r_1}\cdots\lambda^{r_u}.$$

As we have seen, each $A_i$ can be written as a sum of indicator random variables $A_i=\sum_{S,j} X^{i,j}_{S}$, each $X^{i,j}_S$ indicates the event $E_S^{i,j}$ that the $j$-th of the potential copies of 
 $v^*$-marked $l$-Berge-trees on the vertex set $S$ is present. Then
$$\mathbb{E}_n\left[(A_1)_{r_1}\cdots(A_u)_{r_u}\right]=
\sum_{S_1,\ldots,S_u,j_1,\ldots,j_u}\mathbb{P}_n[E_{S_1}^{1,j_1}\land\ldots
\land E^{u,j_u}_{S_u}].$$
The above sum splits into $\sum_1+\sum_2$ where $\sum_1$ consists of the terms with $S_1,\ldots,S_u$ pairwise disjoint. It is easy to see that if  

$$p=\frac{d!}{v^*}\frac{\log n+l\log\log n+c(n)}{n^d}$$ 
then
             $\sum_1\sim\prod_i\lambda^{r_i}$.
             
             Arguing as in Proposition \ref{intersectiontype}, one sees that the contribution of each of the terms in
     $\sum_2$ with a given pattern of intersection is $o(1)$. Hence $\sum_2=o(1)$ and we are done.
\end{proof}

Note that in this case again, $\mathcal{K}(\Theta_p)$ is countable with a countably infinite number of limit points. Each one 
 corresponds to specifying finite quantities for the various isomorphism types of marked Berge-trees, except for one, whose copies are insisted 
to appear an infinite number of times.


These pieces together prove the following theorem.

\begin{Thrm}

All elements in $\bc$ are convergence laws.

\end{Thrm}

\begin{proof}

Just note that all $L$-functions on the above range must satisfy, with the familiar definitions of 
$\omega(n)$ and $c(n)$, one of the following conditions:

           \begin{enumerate}

                         \item $n^{-d}\ll p\ll(\log n)n^{-d}$
                         \item $(\log n)n^{-d}\ll p\ll n^{-d+\epsilon}$ for all positive $\epsilon$
                         \item $p\sim C\cdot\frac{\log n}{n^d}$, where $\frac{d!}{v^*+1}<C<\frac{d!}{v^*}$ for  some $d,v^*\in\mathbb{N}$
                .         \item $p\sim \frac{d!}{v^*}\cdot\frac{\log n}{n^d}$ where $\omega\to\pm\infty$ or
            $\omega\to C$ where $l-1<C<l$ for some $l\in\mathbb{N}$
                           \item $\omega\to l\in\mathbb{N}$ and $c(n)\to\pm\infty$ or $c(n)\to c\in\mathbb{R}$.
                                                    
            \end{enumerate}
\end{proof}

As it was the case in the last section, it is worth noting that the arguments used in getting zero-one laws for the clauses $1$, $2$ and $3$ do not require the edge functions to be in Hardy's class, so \emph{all} functions inside those intervals are zero-one laws, regardless of being $L$-functions.
 
 On the other hand, taking $\omega(n)$ oscillating infinitely often between two constant values 
 $\omega_1<l$ and $\omega_2>l$ makes the probability of an elementary event oscillate between zero and one. Similarly, taking $c(n)$ oscillating between any two
 different  positive values makes the probability of an elementary event oscillate between two different values $\notin\{0,1\}$. Obviously, these situations rule out convergence laws.

 As it was the case with $\bb$, our present discussion implies that, in a certain sense, most of the functions in $\bc$ are zero-one laws: the only way one of that functions can avoid this condition is being inside
one of the countable windows inside a local threshold for the presence of marked Berge-trees of some order.

                               \section{The Double Jump}

            In this section, we consider the random hypergraph $G^{d+1}(n,p)$, where $p\sim\frac{\lambda}{n^d}$ for some constant $\lambda>0$. Of course, this is equivalent to having 
 $p=\frac{\lambda_n}{n^d}$, where $\lambda_n\to\lambda$.
Our goal is to show that the above $p$'s are convergence laws.            
             
             Simple applications of  Theorem \ref{vantsyan} 
imply that, in this range, the countable models of the almost sure theory have infinitely many 
connected components isomorphic to each finite Berge-tree and no bicyclic (or more) components.
Also, there possibly are infinite Berge-trees as components. As we have already seen, these
components do not matter from a first order perspective, as they are simulated by sufficiently
large finite Berge-trees. More precisely: the addition of components that are infinite Berge-trees do not
alter the elementary type of a hypergraph that has infinitely many copies of each finite Berge-tree.

         The above considerations suggest that it may be useful to consider the asymptotic
 distribution of the various types of unicyclic components in $G^{d+1}(n,p)$. It follows that the 
 distributions are asymptotic independent Poisson. But now, the structure of the set of possible
 completions of the almost sure theory is more complex: it is a Cantor space.


        \subsection{Values and Patterns}


 A \emph{rooted Berge-tree} is simply a Berge-tree $T$ (finite or infinite) with one distinguished vertex $R\in T$, called the \emph{root}. With rooted Berge-trees, the concepts of \emph{parent, child, ancestor} and \emph{descendent}
are clear: their meaning is similar to their natural computer science couterparts for rooted trees. The \emph{depth} of a vertex is its distance from the root. For each $w\in T$, $T^w$ denotes the sub-Berge-tree consisting
of $w$ and all its descendants. 

For $r,s\in\mathbb{N}$, we define the $(r,s)$-value of $T$ by induction on $r$. Roughly speaking, we examine the $r$-neighborhood of $R$ and consider any count greater than $s$, including infinite, indistinguishable 
from each other and call them ``many". Indeed, the possible $(1,s)$-values for a rooted tree $T$ are $0,1,2,\ldots s, \mathcal{M}$, where the symbol $\mathcal{M}$ stands for ``many": the $(1,s)$-value of $T$ is then the number of 
edges incident on the root $R$ if this number is $\le s$; otherwise, the $(1,s)$-value of $T$ is $\mathcal{M}$.

Now suppose the concept of $(r,s)$-value has been defined for all rooted Berge-trees and denote by $\val(r,s)$ the set of all possible such values. In what
follows, $\mathbb{N}_s=\{0,1,\ldots,s,\mathcal{M}\}$. Define, for all $s$ and by induction on $r$,

                                         $$\pat(r,s)=\left\{P:\val(r,s)\to\mathbb{N}_s\mid\sum_{\Gamma\in\val(r,s)}P(\Gamma)=d\right\}$$
and

                                         $$\val(r+1,s)=\left\{\Gamma:\pat(r,s)\to\mathbb{N}_s\right\}.$$                                      
Intuitively, consider an edge $E=\{R,w_1,\ldots,w_d\}$ of $T$ incident on the root $R$.
The pattern of $E$ is the function $P:\val(r,s)\to\mathbb{N}_s$ such that, for all values $\Gamma\in\val(r,s)$, there are exactly $P(\Gamma)$ elements in the set $\{T^{w_1},\dots,T^{w_d}\}$ with $(r,s)$-value
$\Gamma$. 
The $(r+1,s)$-value of $T$ is the function $\Gamma:\text{PAT}(r,s)\to\mathbb{N}_s$ such that, for all
$\Delta\in\pat(r,s)$, the root $R$ has exactly $\Gamma(\Delta)$ edges incident on it with pattern $\Delta$, with $\mathcal{M}$ standing for ``many". 
Note that one always has
$$\sum_{\Gamma\in\val(r,s)}P(\Gamma)=d.$$

For any value $\Gamma\in\val(r,s)$, one can easily create a finite rooted Berge-tree with value $\Gamma$ (for instance, interpret $\mathcal{M}$
 as $s+1$). Also, any rooted Berge-tree can be considered a uniform hypergraph
by removing the special designation of the root.

  \begin{Def}
  
  Fix a vertex $v$ in the random hypergraph $G^{d+1}(n,p)$ and a value $\Gamma\in\val(r,s)$. Then $p_{\Gamma}^n$ is the 
  probability that the ball of center $v$ and radius $r$ is a Berge-tree of value $\Gamma$, considering
  $v$ as the root.
  
  Similarly, for an edge $E=\{v,v_1,\ldots,v_d\}\in G^{d+1}(n,p)$ and a pattern $\Delta\in\pat(r,s)$, $p_{\Delta}^n$ is the probability that the pattern of $E$ is $\Delta$, considering $v$ as the
  root.

  \end{Def}

         Next, we proceed to describe the asymptotic behavior of $p_{\Gamma}^n$ and 
   $p_{\Delta}^n$ as $n\to\infty$.

          \subsection{Poisson Berge-Trees}

       Now we consider a random procedure for constructing a rooted Berge-tree. In fact, it is a simple 
modification of the Galton-Watson Branching Process, aiming to fit the case of hypergraphs.

        $\text{B}(r,\mu)$ is the random rooted Berge-tree constructed as follows:
        
        Let $P(\mu)$ be the Poisson distribution with mean $\mu$ and start with the root $v$. The number of edges incident on $v$ is $P(\mu)$. Each child of $v$ has, in turn, further $P(\mu)$ edges
incident on it (there being no further adjacencies among them, so as to $\text{B}(r,\mu)$ remain
a Berge-tree). Repeat the process until the end of the $r$-th generation and then halt. The resulting structure
is a random Berge-tree rooted on $v$.
           
           One obtains a similar structure $\tilde{\text{B}}(r,\mu)$ beginning with an edge $E$, rooted 
on $v$, and requiring that each non-root vertice has $P(\mu)$ further edges and so on, until the
$r$-th generation.

            For $\Gamma\in\val(r+1,s)$ be a $(r+1,s)$-value, let $p_{\Gamma}$ be the probability that
$B(r+1,s)$ has value $\Gamma$. Similarly, if $\Delta\in\pat(r,s)$ is a $(r,s)$-pattern, let 
$p_{\Delta}$ be the probability that $\tilde{\text{B}}(r,s)$ has pattern $\Delta$. Note that the
method of factorial moments implies that if 
                        $$\pat(r,s)=\{\Delta_1,\Delta_2\ldots,\Delta_N\}$$
then the distributions of edges incident on the root $v\in B(r+1,\mu)$ with pattern $\Delta_i\in\pat(r,s)$  are poisson $P(p_{\Delta_i}\cdot\mu)$, \emph{independently} for each $i\in\{1,\ldots,N\}$.

            The most important piece of information to showing that $p\sim\lambda n^{-d}$ are 
 are convergence laws is that $p_{\Gamma}^n\to p_{\Gamma}$ and $p_{\Delta}^n\to p_{\Delta}$
 for every value $\Gamma$ and every pattern $\Delta$, with $\mu=\frac{\lambda}{d!}$.

           \subsection{Size of Neighborhoods}

        The next lemma shows that the probability that the size of the neighborhood of a given vertice
is large is $o(1)$. There, $|B(v_0,r)|$ is the number of vertices in the ball of center $v_0$ and radius $r$.

\begin{Lemma}   \label{size}

Fix $\epsilon>0$, $\delta>0$ and $r\in\mathbb{N}$. Then

  $$\mathbb{P}_n[\exists v_0 , | B(v_0,r)|>\epsilon n^{\delta}]\to 0.$$

\end{Lemma}

\begin{proof}

We proceed by induction on $r$.

   If $r=0$, there is nothing to prove.

   For $r=1$, first fix $\tilde{\lambda}>\lambda$. Then one has, for sufficiently large $n$,

   \begin{align*}
 &\mathbb{P}_n\left[\exists v_0 , |B(v_0,1)|>\epsilon n^{\delta}\right]\leq n\cdot\mathbb{P}_n\left[|\{\text{neighbors of } v_0\}|>\epsilon n^{\delta}\right]  \\
 &=n\cdot\mathbb{P}_n\left[|\{\text{edges on }v_0\}|>\frac{\epsilon n^{\delta}}{d}\right]
 =n\cdot\sum_{l>\frac{\epsilon n^{\delta}}{d}}\mathbb{P}_n\left[|\{\text{edges on } v_0\}|=l\right] \\
 &\leq n\cdot\sum_{l>\frac{\epsilon n^{\delta}}{d}}\frac{1}{d!}{n\choose{d}}^{l}p^l(1-p)^{n\choose{d}} 
 \leq n\cdot\sum_{l>\frac{\epsilon n^{\delta}}{d}}n^{dl}\cdot\frac{\lambda_n^l}{n^{dl}}\exp(-n^d\cdot\frac{\lambda_n}{n^d}) \\
 &\leq(\text{constant})\frac{\tilde{\lambda}^{\frac{\epsilon n^{\delta}}{d}}}{(\epsilon n^{\delta}d^{-1})!}\cdot\frac{n}{1-\epsilon^{-1}\tilde{\lambda}dn^{-\delta}} \\
& \sim(\text{constant})\frac{\tilde{\lambda}^{\epsilon n^d}\cdot n}{\sqrt{2\pi\epsilon n^{\delta}d^{-1}}\cdot(\epsilon n^{\delta}d^{-1}e^{-1})^{\frac{\epsilon n^{\delta}}{d} }} 
 =o(1) .
   \end{align*}

   For the induction step, note first that the induction hypothesis implies that, almost surely, every ball of
radius $r$ has size at most $\sqrt{\epsilon}n^{\delta/2}$. Let $B$ be the event $\{\exists v_0 , |B(v_0,r+1)|>\epsilon n^{\delta}\}$.
Note that $B\implies\{\exists v_0 , |B(v_0,r)|>\sqrt{\epsilon}n^{\delta/2}\}$.
Therefore $$\mathbb{P}_n[B]\leq\mathbb{P}_n\left[\exists v_0 , |B(v_0,r)|>\sqrt{\epsilon}n^{\delta/2}\right]=o(1).$$
 \end{proof}

In what follows, $\Omega$ is an $(r,s)$-type and, for a hypergraph $C$, $E(C)$
is the edge set of $C$.

 \begin{Lemma}

 Fix vertices $v_1,v_2,\ldots,v_{kd}$ and an unicyclic connected configuration $C$ on the vertice
 set $\{v_1,v_2,\ldots,v_{kd}\}$ (necessarily with $k$ edges).
 Then, given $C$, the probability that the ball $B(v_1,r)$ of center $v_1$ and radius $r$ on 
 $G^{d+1}(n,p)\setminus E(C)$ 
 is a Berge-tree of type $\Gamma$ is $p_{\Gamma}+o(1)$.

  \end{Lemma}            
             
   \begin{proof}
   
   All probabilities mentioned on this proof are conditional on getting $C$.
   
   By induction on $r$, we show that, for all $0\leq\delta<1$, $0<\epsilon$ and $r\in\mathbb{N}$,
   and given $C$, the ball $B(v_1,r)\subseteq G^{d+1}(n-\epsilon n^{\delta},p)\setminus E(C)$ is a Berge-tree of 
   type $\Gamma$ with probability $p_\Gamma+o(1)$, and, similarly, that given an edge $E$ on $v_1$, the pattern of $E$ in $G^{d+1}(n-\epsilon n^{\delta},p)\setminus E(C)$
is $\Delta$ with probability $p_\Delta+o(1)$.   

For $r=1$, let $\mathcal{E}$ be the edge set of 
      $$B(v_1,r)\subseteq G^{d+1}(n-\epsilon n^{\delta},p)\setminus E(C).$$
   Then one has
   
   \begin{align*}
  & \mathbb{P}_n[|\mathcal{E}|=l]\sim\frac{1}{l!}{n\choose d}^l\cdot p^l\cdot(1-p)^{n\choose d} 
   \sim\frac{n^{dl}}{l!(d!)^l}\cdot(\frac{\lambda}{n^d})^l\cdot\exp\left(-p\frac{n^d}{d!}\right) \\
   &=\frac{1}{l!}(\frac{\lambda}{d!})^l\cdot\exp\left(-\frac{\lambda}{d!}\right)
   \end{align*}
    so that $\mathcal{E}$ has asymptotic distribution $P(\frac{\lambda}{d!})$, which agrees to $B(1,\frac{\lambda}{d!})$.
    A similar argument applies in the case of patterns of edges.
   
   For the induction step, let $\Delta_1,\ldots,\Delta_N$ be the possible $(r,s)$ patterns of 
   edges. We use the method of factorial moments to show that the distributions of the various
   $(r,s)$-patterns are asymptotically independent Poisson with means 
   $\frac{\lambda^l}{l!}\cdot p_{\Delta_i}$, for $1\leq i\leq N$, which clearly suffices.
   First we show that the expected number $\mathcal{P}$ of pairs of edges $E_i,E_j$ on $v_1$
   with patterns $\Delta_i$ and $\Delta_j$ respectively is asymptotically to
$\frac{\lambda^{2l}}{(l!)^2}\cdot p_{\Delta_i}\cdot p_{\Delta_j}$
.     The general case is similar, with more cumbersome notation.
     
     Let $p_0^n$ be the probability of the event $A$ that the pattern of $E_i$ is $\Delta_i$ and 
    that of $E_j$ of $\Delta_j$. Define $\tilde{p}_0^n$ to be the probability of the event $B$ that
    the pattern of $E_i$ is $\Delta_i$ and the pattern of $E_j$, \emph{not counting the vertices
    already used in $E_i$}. 
    By Lemma \ref{size} and the induction hypothesis, one has $\tilde{p}_0^n\to p_{\Delta_i}\cdot p_{\Delta_j}$. 
          Note that $p_0^n\sim \tilde{p_0}^n$, because any hypergraph on the symetric difference $A\triangle B$ 
    has at least two cycles, so that $\mathbb{P}_n[A\triangle B]=o(1)$. So we have

   $$ \mathbb{E}_n[\mathcal{P}]\sim{n\choose d}^2\cdot p^2\cdot p_0^n 
    \sim(\frac{n^d}{d!})^2\cdot p^2\cdot\tilde{p}_0^n 
    \sim\frac{\lambda^2}{d!}\cdot p_{\Delta_i}\cdot p_{\Delta_j}$$
    and we are done.          
 \end{proof}

Similar arguments show that, defining the $(r,s)$-pattern of a cyclic configuration 
 $C$ in the obvious manner, the distribution of cycles of the various types are asymptotically independent
 Poisson. Therefore, if $\Lambda_1,\ldots,\Lambda_k$ are the possible $(r,s)$-patterns
 of cycles,  
  $\textbf{m}$ is the $k$-tuple $(m_1,\ldots,m_k)\in(\mathbb{N}_s)^{k}$ and $\mathcal{P}_i$ is the number of cycles of pattern $\Lambda_i$,
    then, with the usual interpretation of $\mathcal{M}$ as ``at least $s+1$", the property $\sigma_{\textbf{m}}$ given by
        $(\mathcal{P}_1=m_1)\land\ldots\land(\mathcal{P}_k=m_k)$
    is elementary and the open sets in the family $\{A_{\sigma_{\textbf{m}}}\mid\textbf{m}\in(\mathbb{N}_s)^k, k\in\mathbb{N}\}$ can be organized as the
  nodes of a spanning tree. We have seen above that the probability measure of each such node converges. So the edge
  functions on this range are convergence laws. As none of the branches in this tree is isolated, $\mathcal{K}(\Theta_p)$ has no
  isolated points, therefore being a Cantor space. Thus we get part $(iii)$ of \ref{grandeteorema}.

   \begin{Prop}
   
   If $p\sim\frac{\lambda}{n^d}$, then $p$ is a convergence law. Moreover, the correponding space of completions $\mathcal{K}(\Theta_p)$ is
   a Cantor space.
   
    \end{Prop}

  \begin{figure}[!htb]
    \centering
     \includegraphics[height=90mm,angle=90]{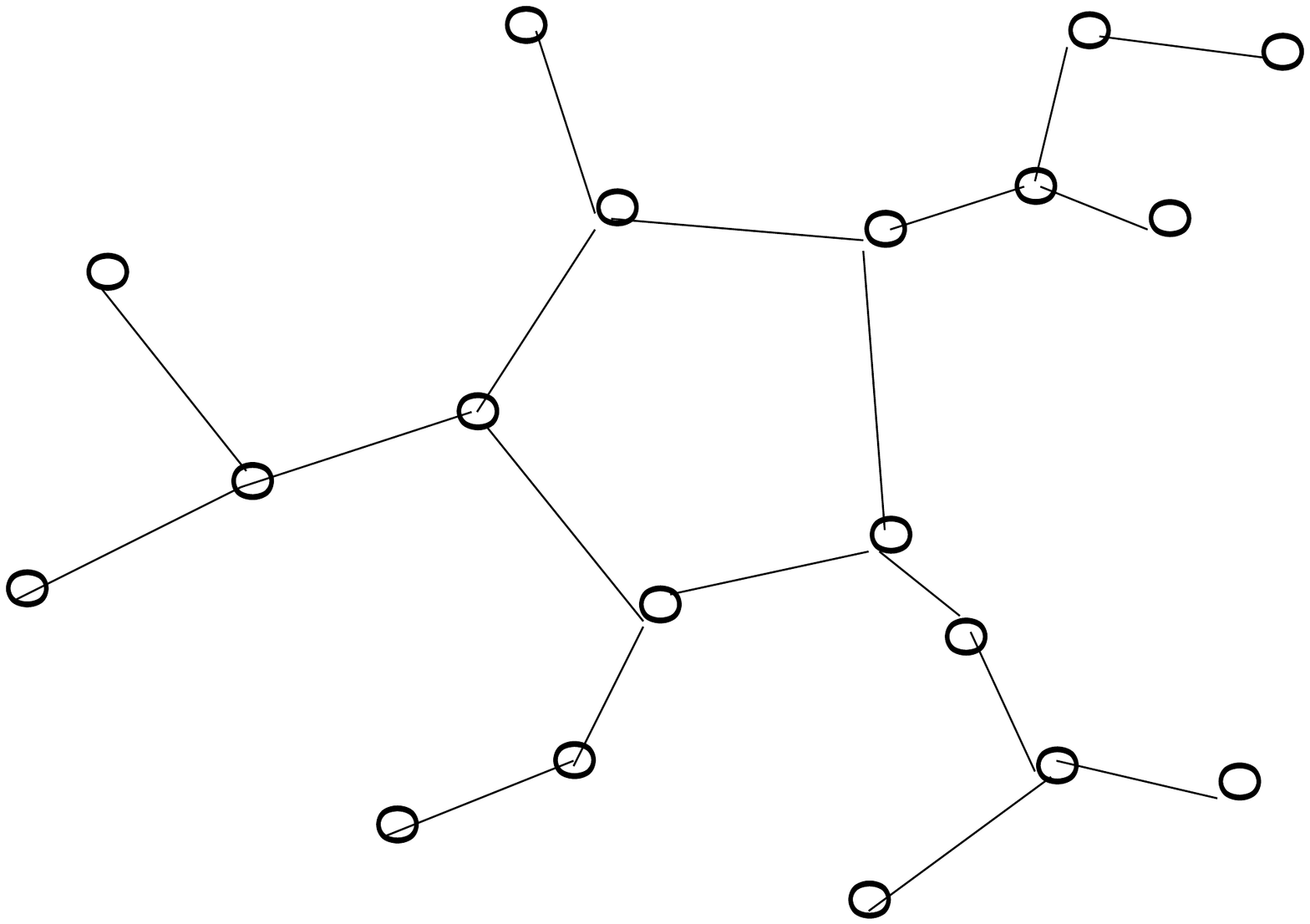}
    \caption{A cyclic type}
    \label{figRotulo4}
  \end{figure}

    The convergence laws in the Double Jump are the last piece of information to obtaining \ref{grandeteorema}.

    
    

\vfill\eject
\input{artigobigcrunch.bbl}



\noindent
Nicolau C. Saldanha, PUC-Rio, saldanha@puc-rio.br \\
M\'arcio Telles, UERJ and PUC-Rio, marcio.telles@uerj.br   \\
Departamento de Matem\'atica, PUC-Rio \\
R. Marqu\^es de S. Vicente 225, Rio de Janeiro, RJ 22453-900, Brazil  \\
Instituto de Matem\'atica e Estat\'istica, UERJ \\
R. S\~ao Fco. Xavier 524, Rio de Janeiro, RJ 20550-900, Brazil \\

\end{document}

%% file: artigobigcrunch.bbl
\providecommand{\bysame}{\leavevmode\hbox to3em{\hrulefill}\thinspace}
\providecommand{\MR}{\relax\ifhmode\unskip\space\fi MR }
\providecommand{\MRhref}[2]{%
  \href{http://www.ams.org/mathscinet-getitem?mr=#1}{#2}
}
\providecommand{\href}[2]{#2}